\documentclass{amsart}

\usepackage{graphicx}

\usepackage{amsmath, amssymb, array, latexsym, hyperref, xcolor}

\usepackage{yfonts}

\usepackage{fullpage}

\newtheorem{theorem}{Theorem}[section]
\newtheorem{lemma}[theorem]{Lemma}
\newtheorem{proposition}[theorem]{Proposition}
\newtheorem{corollary}[theorem]{Corollary}

\theoremstyle{definition}
\newtheorem{definition}[theorem]{Definition}

\theoremstyle{remark}
\newtheorem{remark}[theorem]{Remark}

\numberwithin{equation}{section}

\newcommand{\prend}{$\hfill \Box$}
\newcommand{\ls}{\leqslant}
\newcommand{\gr}{\geqslant}

\def\R{{\mathbb R}}
\def\supp{{\rm supp\,}}
\def\Exp{{\mathbb E\,}}
\def\Var{{\rm Var}}
\def\Prob{{\mathbb P}}
\def\Lip{{\rm Lip}}
\def\Event{{\mathcal E}}
\def\unc{{\rm unc}}
\def\semi{{\Upsilon}}

\begin{document}

\title{Hypercontractivity and lower deviation estimates in normed spaces}

\author{Grigoris Paouris}
\address{Department of Mathematics, Mailstop 3368, Texas A\&M University, College Station, TX, 77843-3368.}  
\email{grigorios.paouris@gmail.com}
\thanks{The first author was supported by the NSF grant DMS-1812240.}

\author{Konstantin Tikhomirov}
\address{School of Mathematics, GeorgiaTech, Atlanta, GA, 30332.}
\email{ktikhomirov6@gatech.edu}
\thanks{The second author was partially supported by the Simons foundation.}

\author{Petros Valettas}
\address{Mathematics Department, University of Missouri, Columbia, MO, 65211.}
\email{valettasp@missouri.edu}
\thanks{The third author was supported by the NSF grant DMS-1612936 and by Simons Foundation grant 638224.}


\date{}

\begin{abstract}
We consider the problem of estimating small ball probabilities $\mathbb P\{f(G) \ls \delta \mathbb Ef(G)\}$ for sub-additive,
positively homogeneous functions $f$ with respect to the Gaussian measure. We establish estimates 
that depend on global parameters of the 
underlying function which take into account analytic and statistical measures, such as the variance and the $L^1$-norms of 
its partial derivatives. 
This leads to dimension-dependent bounds for small ball and lower small deviation estimates 
for seminorms when the linear structure is appropriately chosen to optimize the aforementioned parameters. 
Our bounds are best possible up to numerical constants. In all regimes, $\|G\|_\infty = \max_{ i \ls n}|g_i|$ 
arises as an extremal case in this study. 
The proofs exploit the convexity and hypercontractivity properties of the Gaussian measure. 
\end{abstract}

\maketitle

\tableofcontents

\section{Introduction}

The concentration of measure phenomenon is one of the most important concepts in high-dimensional 
probability and is an indispensable tool in the study of high-dimensional structures that arise in theoretical and 
applied fields. In its most simple form it can be stated as follows: functions depending smoothly on many independent variables have
small fluctuations. More specifically, in Gauss' space, asserts that any function $f$ on $\mathbb R^n$ 
which is $K$-Lipschitz is almost constant with overwhelming probability:
\begin{equation} \label{eq:conc-L}
\max\big( \Prob  \{ f (G) \ls \Exp f(G) - t K \} , \Prob \{ f (G)\gr \Exp f(G) + t  K\} \big) \ls \exp( - t^2 /2 ) , \quad t>0,
\end{equation} where $G$ is the standard Gaussian vector in $\mathbb R^n$, see e.g. \cite{Pis-pmb, Mau-t}. 
This probabilistic phenomenon is usually addressed as consequence of the solution to the Gaussian isoperimetric 
problem, which was solved independently by Sudakov and Tsirel'son \cite{ST} and by Borell \cite{Borell}.
The Gaussian isoperimetric inequality asserts that among all Borel sets in $\mathbb R^n$ of a given
Gaussian measure, the half-spaces have the smallest
Gaussian surface area. 

However, isoperimetry is not the only reason for the concentration phenomenon; 
convex functions are known to share the above property even when they are not Lipschitz. 
In this note we investigate concentration phenomena that appear due to convexity rather than isoperimetry. 
Having said that let us comment on the symmetric nature of \eqref{eq:conc-L}. 
If one establishes the one inequality the other one follows easily by applying it to $-f$, 
since $-f$ is also Lipschitz. This is not of course true in the case of convex functions. 
So we focus only on the “one-sided'' inequalities.

In the present work we study the phenomenon in the ``lower deviation” regime, 
that is $\mathbb P\{ f(G) \ls \delta \mathbb Ef(G)\}, \; 0<\delta<1$, for sub-additive and positively homogeneous functions $f$. Further,
this contains two other important regimes (which we distinguish because frequently they exhibit different behaviors), namely 
the ``small ball regime" when $0<\delta<1/2$, and the (lower) ``small deviation regime'' when $1/2< \delta<1$.
An important special case of the aforementioned functions are the seminorms which are also Lipschitz. 
In the light of \eqref{eq:conc-L}, one gets lower deviation estimates for such functions 
in terms of the Lipschitz constant. However, several key examples show that in the lower deviation 
regime better estimates may hold. A typical example is the $\ell_\infty$-norm which satisfies the following estimate (see e.g., \cite{KV}): 
\begin{align} \label{eq:d-cube}
\exp(-Cn^{1-c\delta^2} )<
\mathbb P\left\{ \|G\|_\infty < \delta \mathbb E\|G\|_\infty \right\} <  \exp(-cn^{1-C\delta^2}), \quad \delta\in \left( \frac{1}{\sqrt{\log n}}, \frac{1}{2} \right),
\end{align} whereas a standard application of the Gaussian concentration \eqref{eq:conc-L} would only yield
\[
\mathbb P\left\{ \|G\|_\infty < \delta \mathbb E\|G\|_\infty \right\} <  \exp(-c\log n), \quad \delta\in (0, 1/2).
\] where $c,C>0$ are universal constants.\footnote{We shall make frequent use of the letters $c,C,c_1,\ldots$ throughout the text 
for universal constants whose value may change from line to line. We also use the (standard) asymptotic notation: 
For any two quantities $Q_1,Q_2$ we write $Q_1\lesssim Q_2$ if there exists universal constant $c>0$ such that 
$Q_1\ls cQ_2$. We also  write $Q_2\gtrsim Q_1$ if $Q_1\lesssim Q_2$. Finally, we write $Q_1\asymp Q_2$ if
$Q_1\lesssim Q_2$ and $Q_2\lesssim Q_1$.}

A main goal of this work is to uncover the probabilistic principles which are responsible for this diverse behavior, within
a general context, by determining proper parameters that govern the estimates. On this direction 
our first main result reads as follows:

\begin{theorem} \label{thm:main-1}
	Let $f:\mathbb R^n \to \mathbb [0,\infty)$ be a positively homogeneous\footnote{A function $f: \mathbb R^n\to \mathbb R$ is said to be positively 
	homogeneous if $f(\lambda x)=\lambda f(x)$ for all $\lambda>0$ and $x\in\mathbb R^n$.},
	convex map. Suppose that for some $L>0$ we have
	\begin{align*} 
			\sum_{j=1}^n \left( \mathbb E|\partial_j f(G)|\right)^2
			\ls L (\mathbb Ef(G))^2. 
	\end{align*} Then, we have the following:
	\begin{enumerate}
	
		\item If $\|\nabla f(G)\|_2 \in L^2$ and $\tilde \beta= \mathbb E\|\nabla f(G)\|_2^2/ (\mathbb E f(G))^2$, then
		for any $\delta \in (0,1/2)$ we have
			\begin{align} \label{eq:main-1-1}
				\mathbb P \left\{ f(G) \ls \delta \mathbb E[f(G)] \right\} \ls 
				\exp\left\{ -c\delta^2 \left( \frac{1}{\tilde \beta}\right )^{\tau(\delta)} \left( \frac{1}{L}\right)^{1-\tau(\delta)} \right\},
			\end{align} where $\tau(\delta) \asymp \delta^2$.
				
		\item If $f(G)\in L^2$ and $\beta = {\rm Var}[f(G)] / (\mathbb E f(G))^2$, then for any $\delta\in (0,1/2)$ we have
			\begin{align} \label{eq:main-1-2}
				\mathbb P \left\{ f(G) \ls \delta \mathbb E[f(G)] \right\} \ls 
				\exp\left\{ -c\delta^2 \left( \frac{1}{ \beta}\right )^{\omega (\delta)} \left( \frac{1}{L}\right)^{1-\omega (\delta)} \right\},
			\end{align} where $\omega (\delta) \asymp \delta$.
	\end{enumerate} 
\end{theorem} 

Let us comment on the above result. Note that under homogeneity, 
the sub-additivity and the convexity are equivalent notions, thus the functions under discussion
are locally Lipschitz and the partial derivatives exist a.e. by Rademacher’s theorem. Some more remarks 
are in order. The parameter $L$ depends on the choice of the orthonormal basis with respect to which 
the partial derivatives are considered. Also, let us emphasize that 
the assertion \eqref{eq:main-1-1} provides better dependence with respect to $\delta$ compared to 
\eqref{eq:main-1-2}, however the parameter $\tilde \beta$ is smaller than $\beta$. In particular, 
we have the following inequalities:
\[
\beta \ls \tilde \beta \ls 2\pi.
\]
The estimate $\beta \ls \tilde \beta$ follows from the Gaussian Poincar\'e inequality \cite{Chen} 
which asserts that any absolutely continuous function $f$ 
satisfies
\begin{align} \label{eq:Poin}
	{\rm Var}[f(G)] \ls \mathbb E \|\nabla f(G)\|_2^2.
\end{align} The estimate $\tilde \beta \ls 2 \pi$ exploits the convexity and the sub-additivity. Indeed; due to these facts we have
\[
f(x) \gr f(x+y)-f(y) \gr \langle \nabla f(y) ,x\rangle,
\] for almost every $y$ and all $x$ in $\mathbb R^n$. Hence, we may write
\[
\mathbb E \|\nabla f(G)\|_2^2 = \frac{\pi}{2} \mathbb E_G \left( \mathbb E_Z |\langle \nabla f(G),Z\rangle| \right)^2 \ls 
\frac{\pi}{2}\mathbb E_G \left( 2 \mathbb Ef(Z) \right)^2 = 2\pi (\mathbb Ef(Z))^2,
\] where $G,Z$ are independent standard Gaussian vectors. 

The reason for 
the additional statement \eqref{eq:main-1-2} is due to the fact that $\beta$ can be significantly smaller than
$\tilde \beta$, i.e., up to a factor that depends on the dimension of the ambient space. In order to illustrate
that let us revisit the case of the $\ell_\infty$ norm. It is known (see e.g., \cite{Cha}) that 
\begin{align}
	{\rm Var}[ \|G\|_\infty ] \asymp \frac{1}{\log n}, \quad \| \nabla \|G\|_\infty \|_2 =1, \quad \mathbb E \| G \|_\infty \asymp \sqrt{\log n},
\end{align} hence $\beta(\|\cdot\|_\infty) / \tilde \beta (\|\cdot\|_\infty) \asymp (\log n)^{-1}$.
Following Chatterjee \cite{Cha} the function $x\mapsto \| x \|_\infty$ is {\it superconcentrated} since
\[
{\rm Var}[\|G\|_\infty] \ll \mathbb E\|\nabla \|G\|_\infty \|_2^2. 
\]
That said \eqref{eq:main-1-2} becomes significant for ``superconcentrated" functions $f$
(see section 2 for further details and precise definitions). However, in this specific example the superconcentration
phenomenon is rather poor (in fact \eqref{eq:main-1-2} is worse than \eqref{eq:main-1-1} for all $\delta$) to highlight the usefulness of bound \eqref{eq:main-1-2}. 
Note that, in general, we have
\[
\left( \frac{1}{\beta} \right)^{\omega(\delta)} \left(\frac{1}{L}\right)^{1- \omega(\delta)} \gg  \left( \frac{1}{\tilde \beta}\right)^{\tau(\delta)} \left( \frac{1}{L} \right)^{1- \tau(\delta)} \quad 
\Longleftrightarrow  
\quad \frac{1}{s(f)} \gg R(f)^{1-\frac{\tau(\delta)}{\omega(\delta)}},
\] where $R(f)$ and $s(f)$ abbreviate the following quantities:
\[
R(f):= \frac{ \sum_{j=1}^{n} \| \partial_{j} f \|_{L^2}^{2} } { \sum_{j=1}^{n}  \| \partial_{j} f \|_{L^1}^{2}} = \frac{\tilde \beta}{L}, \quad 
s(f): = \frac{{\rm Var}[f(G)]}{\mathbb E\|\nabla f(G)\|_2^2} = \frac{\beta}{\tilde \beta}. 
\]
Now the weakness of \eqref{eq:main-1-2} is explained from the estimates $R(\|\cdot\|_\infty)=n$ and $s(\|\cdot\|_\infty) \asymp (\log n)^{-1}$.

Another, more striking, example of superconcentration is the $\|\cdot\|_{\rm op}$ 
for Gaussian matrices which exhibits polynomial
(with respect to the dimension) gap in the Poincar\'e inequality. On this particular 
example one can check that the bound in small ball probability provided by \eqref{eq:main-1-2} is significantly smaller 
than the bound in \eqref{eq:main-1-1} for {\it all} $\delta \in ( 0,c_0)$ for some sufficiently small absolute constant $c_0>0$. 
Indeed, let $f$ be defined on $\mathbb R^{n^{2}}$ by $f(G)=\|G\|_{\rm op}$, 
where $ G=(g_{ij})$ is an $n\times n $ Gaussian matrix (i.e., its entries are i.i.d. standard Gaussian random variables).
Then, $f$ is non-negative, convex, and positively homogeneous (in fact $f$ is a norm). 
In particular, $f(G)^{2}=\|G\|_{\rm op}^2= \lambda_{\max} (G^{\ast} G)$, 
where $ \lambda_{\max}$ stands for the largest eigenvalue. 
Let $u=u_G$ be the eigenvector (of Euclidean norm $1$) that corresponds to the $\lambda_{\max}(G^\ast G)$. 
Clearly, 
\[ G^{\ast} G u = \lambda_{\max} u = f(G)^2 u \quad 
 \Longrightarrow  
\quad f^{2} (G) = \langle G^{\ast} G u, u \rangle = \sum_{i=1}^{n} \left( \sum_{j=1}^{n} g_{ij} \langle u, e_{j}\rangle \right)^{2}.
\]
Differentiation with respect to $g_{ij}$ yields $f(G) \partial_{ij} f(G) =  \langle u, e_{j} \rangle \langle Gu , e_{i} \rangle$, which in turn implies 
that 
\[
\sum_{i,j=1}^n|  \partial_{ij} f(G)|^{2} = 1, \quad  \textrm {a.s.}
\] 
Moreover, $u$ is uniformly distributed on $S^{n-1}$ and 
$G\stackrel{d} = UG$ for any orthogonal transformation $U$ 
(these assertions follow from the invariances of the Gaussian matrix). Hence, one may infer that 
\[
\mathbb E|\partial_{ij}f(G)| = \mathbb E \left[ \frac{|\langle u,e_j\rangle|}{f(G)} |\langle Gu, e_i\rangle| \right] = (\mathbb E|\langle u,e_j\rangle )|^2 \asymp \frac{1}{n}.
\] It follows that $R(\|\cdot\|_{\rm op})\asymp 1$. On the other hand,
known estimates on the variance of the largest eigenvalue (see e.g. \cite{LedRid}) yield 
\[ 
{\rm Var} [f(G)^2] ={\rm Var}[\lambda_{\max}(G^\ast G)] \ls Cn^{2/3} \quad \Longrightarrow \quad {\rm Var}[f(G)]\ls C'n^{-1/3}  \quad \Longrightarrow \quad s(\|\cdot\|_{\rm op})\ls C'n^{-1/3},
\] where we have used the fact that $\mathbb E\|G\|_{\rm op}\asymp \sqrt{n}$ and 
the inequality $(\mathbb E\xi)^2 {\rm Var}(\xi) \ls {\rm Var}(\xi^2)$ which is valid for any non-negative random variable $\xi$. This proves our claim. 

The above discussion shows that the bounds \eqref{eq:main-1-1} and \eqref{eq:main-1-2} are incomparable, hence one may ask if 
\eqref{eq:main-1-2} can be proved with the stronger
\footnote{In this case, of course, \eqref{eq:main-1-1} becomes redundant.} 
dependence $\delta^2$. It turns out that this interesting question is closely related to the optimal constant
in Talagrand's $L^1-L^2$ bound \cite{Tal-russo}; we elaborate further on this in Section 2. 

Talagrand's inequality is vital for our analysis as it provides the mean for quantifying the superconcentration 
phenomenon. This also demystifies the appearance of the $\|\partial_i f\|_{L^1}$ in our bounds. Further tools
exploited are: the hypercontractivity of the Ornstein-Uhlenbeck semigroup $(P_t)_{t\gr 0}$, the fact 
that convexity is preserved under the action of $(P_t)$, and deviation inequalities that are available
for convex functions and are formally stronger than \eqref{eq:conc-L}, see e.g. \cite{PV-sdi}. 
The argument can be roughly described as follows: First, 
smoothening via the OU-semigroup endows the function with stronger concentration properties, 
e.g., ${\rm Var}[P_tf(G)]$ decays exponentially fast with time. 
In addition, $P_tf$ well approximates $f$ and interpolates between $f=P_0f$ and 
$\mathbb E[f(G)]=P_\infty f$ (due to ergodicity), while the expectation remains unaltered during the ensuing motion, 
i.e. $\mathbb EP_t f(G) =\mathbb Ef(G)$, for all $t\gr 0$. This allows for replacing $f$ by $P_tf$ and 
reduce the problem of estimating $\mathbb P\{f(G) \ls \delta \mathbb Ef(G)\}$ to a
deviation for $P_tf$. Next, the application of the aforementioned (stronger) deviation
estimate for convex functions is not loose in the small deviation regime for $P_tf$, after specific time $t$. 
In other words, $P_t$ ``lifts" $f$ in order to ``slide"
the small ball regime of $f$ to the lower small deviation regime of $P_tf$. Roughly speaking, we 
arrive at an inequality of the form
\[
\mathbb P\{f(G) < \delta \mathbb Ef(G)\} \ls \mathbb P\{P_tf(G) <\varepsilon(t,\delta) \mathbb EP_tf(G)\},
\] for which one would like to choose $t$ as large as possible, when $P_tf$ has almost no deviation. The admissible 
range for $\varepsilon(t,\delta)$ determines the specific time $t$ that has to be chosen with respect to $\delta$ in order to obtain the desired
estimate.

\medskip

Being interested in dimension-dependent bounds we have to choose the linear structure
appropriately in order to optimize the order of magnitude of the parameters (in terms of the dimension) that
govern our probabilities. To this end, we consider three positions (linear images of the norm) 
which will serve for this purpose. The first two are well known in the theory of Banach spaces \cite{Pis-book}, while the latter is 
dictated by the bounds on hand and seems to be relatively new. It was first introduced in this context in \cite{PV-dicho}
for settling the problem of optimal dependence in the randomized Dvoretzky theorem.
Recall that a norm $f$ in $\mathbb R^n$ is said to be in {\it $\ell$-position} if it satisfies
\[
\mathbb E[ \langle G,e_i \rangle \partial_if(G) f(G) ]= \frac{\mathbb E(f(G))^2}{n}, \quad i=1,2,\ldots,n. 
\] 
A variant of this position is the so-called {\it position of minimal $M$} which is described by a similar 
balancing condition
\[
\mathbb E[ \langle G,e_i \rangle \partial_if(G) ]= \frac{\mathbb Ef (G) }{n}, \quad i=1,2,\ldots,n.
\]
We refer the reader to \cite{GM} for a detailed discussion on the isotropic conditions of these
positions. Finally, a position, which can be naturally addressed in view of our parameter $L$, is the following:
We will say that a norm $f$ in $\mathbb R^n$ satisfies the {\it $w^{1,p}$--condition} 
\footnote{We name the position after the standard notation is used in Sobolev norms and spaces.} if
\[
\mathbb E |\partial_i f(G)|^p = \mathbb E |\partial_j f(G)|^p, \quad i,j=1,\ldots,n,
\] for $p>0$ (of particular interest for us is the case $p=1$). 

Another important structural notion which will be crucial in our arguments 
is the {\it unconditionality}. A norm $\|\cdot\|$ in $\mathbb R^n$ 
is said to be $1$--unconditional, if it satisfies $\|\sum_{i=1}^n \varepsilon_i \alpha_i e_i\| = \|\sum_{i=1}^n \alpha_i e_i\|$ for all 
scalars $(\alpha_i) \subset \mathbb R$ and any choice of signs $\varepsilon_i=\pm 1$. 

With this terminology and by employing Theorem \ref{thm:main-1} we prove our second main result:

\begin{theorem}[Small ball estimates for norms]\label{thm:main-2}
Let $\|\cdot \| $ be a norm in $\R^{n}$.
\begin{enumerate}
\item If $\|\cdot\|$ is $1$--unconditional, in the $M$--position, or in the $\ell$--position, or $w^{1,1}$--position, then for any $ \delta\in ( 0,1/2]$ we have
\begin{equation} \label{eq:main-2-1}
\Prob \left\{ \| G \| \ls \delta \Exp \| G \| \right\} \ls 2\exp \left( - cn^{1- C \delta^2} \right) , 
\end{equation}
where $C,c>0$ are universal constants. 
\item If $\|\cdot\|$ is an arbitrary norm in $\mathbb R^n$, then there exists a linear invertible map $T$ such that for every 
$ \delta\in ( 0,1/2]$ we have
\begin{equation} \label{eq:main-2-2}
\Prob \left\{ \| TG \| \ls \delta \Exp \| TG \| \right\} \ls 2\exp \left( - cn^{1/4 - C \delta^2} \right) , 
\end{equation} 
where $C,c>0$ are universal constants. 
\end{enumerate}
\end{theorem}

Let us mention that the aforementioned result in the 1-unconditional case is sharp (up to universal constants) as
the example of the $\ell_\infty$-norm shows; see \eqref{eq:d-cube}. In \eqref{eq:main-2-2} 
the existence of the linear map $T$ is established by employing topological methods (Borsuk--Ulam theorem) 
and combinatorial tools based on a dichotomy that exploits a fundamental theorem of Alon and Milman \cite{AM}.
This technique has been developed in \cite{PV-dicho} and plays 
significant role in the present work, too. Unlike to the unconditional case, the exponent $1/4$ that we obtain is 
probably far from being optimal. However any previous known result to us had only polynomial 
dependence on the dimension in the probability \eqref{eq:main-2-2}. 
To the best of our knowledge this is the first result that establishes small ball probabilities for norms 
(in some position) with exponential (with respect to the dimension) decay. 
Let us emphasize that the corresponding large deviation inequality 
$\mathbb P \{ \| T G\| \gr (1+t) \mathbb E \|TG\| \} $ can hold only with polynomial dependence $n^{-c\max\{t,t^2\}}$ 
as the example of $\| \cdot \|_{\infty} $ shows; see \cite{Sch-cube,PV-dicho} for details.

A variant of the method we develop to attack the problem in the small ball regime
enables us to encounter its small deviation counterpart as well. To this end, 
the smoothening via the Ornstein-Uhlenbeck semigroup is still involved but the main difference now is that instead of applying 
$P_t$ directly to the norm $f$ we apply it to an appropriately chosen norming set of $f$. 
We perform a step-by-step procedure (algorithm) based on smoothening and elimination of spiky parts of the norm, which
requires at most polynomial number of steps, and allows to deform the original norm 
to one which almost preserves the mean and enjoys better variance bounds. The final stage of this procedure
is completed by an application of the variance-sensitive deviation inequality for convex function from \cite{PV-sdi}.
Following the aforementioned scheme we are able to settle the problem of small deviation for norms in full generality. 
Namely, we prove the following:

\begin{theorem}[Small deviations for norms]\label{thm:main-3}
Let $\|\cdot \| $ be a norm in $\mathbb R^n$.
	\begin{enumerate}
	\item If $\|\cdot\|$ is $1$-unconditional, in the $M$--position, or in the $\ell$--position, or $w^{1,1}$--position, then for
	any $ \varepsilon\in ( 0,1/2)$ we have
		\begin{equation} \label{eq:main-3-1}
		\Prob \left\{ \| G \| \ls (1-\varepsilon) \Exp \| G \| \right\} \ls 3\exp \left( - n^{c\varepsilon} \right) , 
		\end{equation}
	where $c>0$ is a universal constant. 
	\item In the general case, there exists an invertible linear map $T$ such that for any $\varepsilon\in (0,1/2)$ we have
		\begin{align} \label{eq:main-3-2}
		\mathbb P\{ \|TG\| \ls (1-\varepsilon) \mathbb E\|TG\| \} \ls 3\exp \left( -n^{c\varepsilon} \right),
		\end{align} where $c>0$ is a universal constant. 
	\end{enumerate}
\end{theorem}

At this point we want to stress the fact that the above estimates are optimal (up to constants) where the extremal
case occurs (again) for the $\|\cdot\|_\infty$. For the related estimate in the $\ell_\infty$ we refer the reader to \cite{Sch-cube}.

\medskip

Our initial motivation to study this type of questions stems from the attempt to gain better understanding
for the local Euclidean structure in high-dimensional normed spaces.
Notably, the first groundbreaking application of the concentration of measure in the asymptotic theory of normed
spaces is the seminal work of V. Milman \cite{Mil-dvo}
which provides a randomized version of the celebrated theorem of Dvoretzky \cite{Dvo} on almost spherical sections. Applying
\eqref{eq:conc-L} for $f$ being a norm $\|\cdot\|$ (and for $t=\varepsilon \mathbb E\|G \|, \; \varepsilon>0$) we obtain
\begin{align} \label{eq:conc-dvo}
\max\big( \Prob  \{ \|G\| \ls (1-\varepsilon) \Exp \|G\|  \} , \Prob \{ \|G\| \gr (1+\varepsilon) \Exp \|G\|\} \big) \ls \exp( - \tfrac{1}{2} \varepsilon^2 k) , \quad \varepsilon>0,
\end{align} where $k=k(\|\cdot\|):= (\mathbb E \|G\| /{\rm Lip}(\|\cdot \|))^2$ is usually referred to as the {\it critical dimension} or {\it Dvoretzky number} of $\|\cdot\|$.
Then, a standard net argument on the Euclidean sphere yields the existence of ``many" subspaces $E$ (in fact, the random with 
respect to the Haar measure on the Grassmannian) of dimension $m\asymp k(\|\cdot\|)$ that satisfy
\begin{align} \label{eq:dvo-2-sides}
\frac{c\mathbb E\|G\|}{\sqrt{n}} \|y\|_2 \ls \| y \| \ls \frac{C\mathbb E \|G\| }{\sqrt{n}} \|y\|_2, \quad y\in E.
\end{align} The latter applied for $\|\cdot\|_\infty$ and taking into account $k(\|\cdot\|_\infty)\asymp \log n$ yields that 
the $\ell_\infty^n = (\mathbb R^n, \|\cdot\|_\infty)$ admits $k$-dimensional subspaces, with $k \asymp \log n$, 
which are $C$-isomorphic to the Euclidean space $\ell_2^k = (\mathbb R^k , \|\cdot\|_2)$ and this is optimal as was 
shown in \cite{FLM}.

The randomized Dvoretzky theorem, as was put forth by V. Milman in \cite{Mil-dvo}, 
was also the motivation in the work of Klartag and Vershynin \cite{KV}, who established the 
following remarkable phenomenon: If one is interested only in the lower estimate of \eqref{eq:dvo-2-sides}
then subspaces of higher dimension may occur. Namely, they have shown that there exists a parameter $d=d(\|\cdot\|)$, which
is a priori at least as large as the Dvoretzky number, and has the following property: if $m\ls cd$ then for the random $m$-dimensional 
subspace $E$ one has
\begin{align} \label{eq:dvo-1-side}
\frac{c\mathbb E\|G\|}{\sqrt{n}} \|y\|_2 \ls \| y \| , \quad y\in E.
\end{align}
Quite remarkably the parameter $d(\|\cdot\|)$ is intimately connected to the exponential bound in the small ball probability
$\mathbb P\{ \|G\| \ls \delta \mathbb E\|G\|\}$. More precisely, Klartag and Vershynin associated with any norm $f$ the following
parameter:
\begin{align} \label{eq:KV-1}
d(f,\delta):= -\log  \mathbb P \left \{ f(G) \ls \delta {\rm med}(f(G)) \right \},  \quad \delta\in (0,1),
\end{align} where ${\rm med}(\cdot)$ stands for the median, and they proved the aforementioned
lower $\ell_2$-estimate for $d=d(\|\cdot \|, 1/2)$. Building on their work and employing Theorem \ref{thm:main-2} we
are able to give polynomial (as opposed to the previously known logarithmic) bounds of (almost) optimal order for the 
parameter $d(f,\delta)$. These estimates lead to the following one-sided randomized Dvoretzky theorem:

\begin{theorem} [lower $\ell_2$--estimates] \label{thm:main-4}
Let $\|\cdot\|$ be a norm in $\mathbb R^n$.
\begin{enumerate}
	\item If $\|\cdot\|$ is $1$-unconditional in the $\ell$-position or $w^{1,1}$--position, then 
	for any $\delta\in (0,1/2)$ and $k\ls cn^{1-C\delta^2}$, the random $k$-dimensional subspace $E$ of $\mathbb R^n$ satisfies
	\begin{align} \label{eq:main-4-1}
	\|y\| \gr \frac{c \delta^2 \mathbb E\|G\|}{\sqrt{n}} \|y\|_2, \quad y\in E ,
	\end{align} with probability greater than $1-e^{-cn^{1-C\delta^2}}$.
	
	\item In the general case, there exists an invertible linear map $T$ with the following property: for any
	$\delta\in (0,1/2)$ and $m \ls cn^{1/4-C\delta^2}$, the random $m$-dimensional subspace $F$ of $\mathbb R^n$ satisfies
	\begin{align} \label{eq:main-4-2}
	\|Ty\| \gr \frac{c \delta^2 \mathbb E\|TG\|}{\sqrt{n}} \|y\|_2, \quad y\in F,
	\end{align} with probability greater than $1-e^{-cn^{1/4-C\delta^2}}$, where $C,c>0$ are universal constants. 
\end{enumerate}
\end{theorem}

\medskip

The rest of the paper is organized as follows: In Section 2 we recall the main tools for Gaussian measure and lay them out carefully, 
in order to assemble them into the proof of the general lower deviation estimates, thus establishing 
Theorem \ref{thm:main-1}. We provide two variants
of the method:  Theorem \ref{thm:hyper-sb} and Theorem \ref{thm:super-sb}.

In Section 3 we discuss applications of the aforementioned lower deviations
in the framework of normed spaces, establishing Theorem~\ref{thm:main-2} (Theorem \ref{thm:1-unc}, 
Theorem \ref{p: optimal in M}, Theorem \ref{thm:main-2a}), and Theorem \ref{thm:main-4} (Theorem \ref{thm:main-4a}). 
Finally, in Section 4 we prove Theorem \ref{thm:main-3} (Theorem \ref{th: smdev gen} and Theorem \ref{th: smdev unc}). 

\medskip

{\bf Acknowledgment.}  Part of this  work was conducted while the authors 
were in residence at the Mathematical Sciences Research Institute in Berkeley, California, supported by NSF grant
DMS-1440140. The hospitality of MSRI and of the organizers of the program on
Geometric Functional Analysis is gratefully acknowledged.
The authors are also grateful to an anonymous referee whose comments improved the style of this exposition.

\section{Lower deviations with respect to global parameters}\label{s: Small-ball probability estimates}

In this section we introduce a general method (see \S2.4) for proving small ball probability estimates for convex, 
positively homogeneous functions. We provide two variants of the method. 
The first uses the hypercontractivity property of the Gaussian measure;
the second is based on the superconcentration phenomenon, whereas
both of them depend heavily on the convexity properties of the underlying function. 
The convexity ensures that the function deviates less below the median and this 
becomes even more drastic when the function is proportionally smaller than its mean. 
We begin with the background material on the Gaussian tools that we will need for proving the aforementioned results.
Let us emphasize that all functions considered in the sequel are convex, therefore locally Lipschitz. 
Hence, by Rademacher theorem they are differentiable a.e. Therefore, the quantities
$\mathbb E\| \nabla f(G)\|_2^2$, $\mathbb E|\partial_i f(G)|, \ldots,$ are meaningful
without further smoothness assumption.

\subsection{Gaussian deviation inequalities for convex functions} 

For context let us recall the fact that for any convex function the deviation below the 
median is smaller than the deviation above the median with respect to the 
Gaussian measure $\gamma_n$:
\begin{align} \label{eq:skew}
\mathbb P \{ f(G) \ls {\rm med}( f) -t\} \ls \mathbb P\{ f(G)\gr {\rm med}(f) +t \}, \quad t>0, \quad G\sim N({\bf 0},I_n).
\end{align} A rigorous proof of this fact can be given by employing Ehrhard's inequality \cite{Ehr} (see \cite{V-tight} for the details). 
However, the inequality \eqref{eq:skew} 
does not give any information how this skew behavior can be quantified. To this end, there exist one-sided deviation 
inequalities which improve upon the classical 
deviation inequality \eqref{eq:conc-L}, when the function under consideration is additionally convex. 

An inequality of this type can be traced back to the works of Samson \cite{Sam} and 
Bobkov-G\"otze \cite{BG} (see also \cite[\S 5.2]{PV-var}). It essentially states that one may replace in the lower tail in \eqref{eq:conc-L} 
the Lipschitz constant of $f$ by $(\mathbb E \|\nabla f(G)\|_2^2)^{1/2}$. Note that for any Lipschitz function one has 
\begin{align} \label{eq:grad-lip}
\mathbb E\|\nabla f(G)\|_2^2 \ls {\rm Lip}(f)^2.
\end{align} 
The precise statement is the following:

\begin{theorem}[isoperimetry+convexity] \label{thm:sdh}
For any convex map $f:\mathbb R^n \to \mathbb R$ with $\|\nabla f\|_2\in L^2(\gamma_n)$, we have
\begin{align} \label{eq:sdh}
\mathbb P \left\{  f(G) \ls \mathbb Ef(G)  -t \left(\mathbb E\|\nabla f(G)\|_2^2 \right)^{1/2}  \right\} \ls e^{-t^2}, \quad t>0,
\end{align} where $G$ is the standard Gaussian vector in $\mathbb R^n$.
\end{theorem}

Let us mention that this inequality is obtained by melting together the Gaussian isoperimetry and the convexity properties of the 
function. In fact, it holds true for a much larger class of measures, e.g. for all measures satisfying a logarithmic Sobolev inequality. 
We refer the reader to \cite[Chapter 6]{Led-book} and the references therein for related notions.

Another inequality of this type was recently established in \cite{PV-sdi}. 
For convex functions, one may replace the $L^2$-norm of the gradient by 
the variance of the function. More precisely, we have the following:

\begin{theorem} [convexity+convexity] \label{thm:sdi}
Let $f$ be a convex map in $L^2(\gamma_n)$. Then, we have
\begin{align} \label{eq:sdi}
\mathbb P \left\{ f(G) \ls \mathbb Ef(G) - t \sqrt{{\rm Var} [f(G)]} \right \}  \ls e^{-t^2/100}, \quad t>0,
\end{align} where $G$ is the standard Gaussian vector in $\mathbb R^n$.
\end{theorem}

Note that, in the light of the Gaussian Poincar\'e inequality \eqref{eq:Poin}
the estimate \eqref{eq:sdi} improves considerably upon \eqref{eq:sdh}. Moreover, there exist classical examples which indicate
that the variance can be dramatically smaller than $\mathbb E\|\nabla f(G)\|_2^2$ or ${\rm Lip}(f)^2$ 
(see \cite{PVZ}, \cite{LT}, \cite{V-tight}, \cite{Cha}). 
Let us emphasize the fact that, the aforementioned variance-sensitive inequality is strongly connected 
with the Gaussian convexity, via Ehrhard's inequality, and it is not known if holds true for other than Gaussian-like distributions, 
see \cite[Theorem 2.2]{PV-sdi}, \cite[Theorem 5.6]{PV-var} and \cite[\S 2.1.3]{V-tight}. 

In many problems, one is interested in studying the small deviation from the mean (or the median) of a positive convex function $f$, that is
\begin{align*}
\mathbb P \{ f(G) \ls (1-\varepsilon ) \mathbb E[f(G)] \}, \quad 0<\varepsilon<1.
\end{align*} This probability can be estimated by using either \eqref{eq:sdh} or \eqref{eq:sdi} to obtain
\begin{align}
\mathbb P \{ f(G) \ls (1-\varepsilon ) \mathbb E[f(G)] \}\ls \exp\left( -\varepsilon^2/ \tilde \beta(f) \right), \quad 0<\varepsilon<1, \quad  
\tilde \beta(f) :=\frac{\mathbb E\|\nabla f(G)\|_2^2}{(\mathbb E[f(G)])^2},
\end{align} in the first case, and 
\begin{align}
\mathbb P\{ f(G) \ls (1-\varepsilon ) \mathbb E[f(G)] \} \ls 2\exp\left( -c \varepsilon^2 /  \beta(f) \right), \quad 0<\varepsilon<1, \quad  
\beta(f) :=\frac{{\rm Var}[f(G)]}{(\mathbb E[f(G)] )^2},
\end{align} in the latter case. If $f$ is Lipschitz, by taking into account \eqref{eq:Poin} and \eqref{eq:grad-lip} we infer that 
\begin{align} \label{eq:beta-comp}
\frac{(\mathbb E[f(G)])^2}{{\rm Lip}(f)^2} = k(f) \ls \frac{1}{\tilde \beta (f)} \ls \frac{1}{\beta(f)},
\end{align} where $k(f)$ is the Dvoretzky number of $f$ and $\beta(f)$ is referred to as the normalized variance (see e.g. \cite{PV-sdi}). 
Furthermore, if $f$ is additionally a norm one may easily check (see e.g. \cite{PV-var}) that
\begin{align} \label{eq:beta-magni}
n \asymp k(\|\cdot\|_2) \gr k(f) \gr k( |\langle \cdot, \theta\rangle|) = 2/\pi , \quad \beta(f) \gr \beta(\|\cdot\|_2) \asymp 1/n,
\end{align} where $\theta$ is some (any) unit vector.

\subsection{Talagrand's $L^1-L^2$ bound}

As we have already reviewed in the Introduction there 
exist several key situations where ${\rm Var}[f(G)] \ll \mathbb E\|\nabla f(G)\|_2^2$ (e.g., $f(x)= \|x\|_\infty$).
Following Chatterjee \cite[Definition 3.1]{Cha}, we will say that an absolutely continuous 
function $f:\mathbb R^n\to \mathbb R$ is 
{\it $\varepsilon_n$-superconcentrated}, for some $\varepsilon_n \in (0,1)$, if we have
\begin{align}
{\rm Var}[f(G)] \ls \varepsilon_n \mathbb E\|\nabla f(G)\|_2^2.
\end{align}
In view of the above, we may define the superconcentration constant of $f$ (see also \cite{V-tight}) by
\begin{align}
s(f) := \frac{{\rm Var} [f(G)] }{\mathbb E\|\nabla f(G)\|_2^2}.
\end{align}

Although superconcentration occurs quite often, unfortunately, not many methods for establishing it are available.
Hence, it is of great importance to develop general approaches for quantifying this phenomenon efficiently. A way for proving 
superconcentration is via Talagrand's inequality \cite{Tal-russo}, 
which improves upon the classical Poincar\'e inequality \eqref{eq:Poin}.

\begin{theorem} [Talagrand] \label{thm:Tal-L1-L2} 
For any absolutely continuous function $f:\mathbb R^n \to \mathbb R$, we have
\begin{align} \label{eq:Tal-L1-L2}
{\rm Var}_{\gamma_n} (f) \ls C \sum_{i=1}^n \frac{\|\partial_i f\|_{L^2}^2}{ 1+ \log \left(\|\partial_i f\|_{L^2} / \|\partial_i f\|_{L^1} \right)},
\end{align} where $\partial_i f$ stands for the i-th partial derivative of $f$.
\end{theorem}

In the study of the asymptotic theory of finite-dimensional normed spaces, Talagrand's inequality and 
the superconcentration phenomenon were put forward in \cite{PVZ}. The authors there, invoke \eqref{eq:Tal-L1-L2} 
to prove that the $\ell_p$-norms are superconcentrated for $p>\log n$ and to establish sharp concentration
inequalities for this range of $p$. Soon after, the inequality was used in \cite{Tik-unc} for proving
that every 1-unconditional norm, in $\ell$-position, has concentration at least as good as the $\ell_\infty$-norm. 
Subsequently, it was exploited in \cite{PV-dicho}, to show that every norm admits an invertible linear 
image with concentration at least as good as the $\ell_\infty$-norm. 

In the present work we will use Talagrand's inequality, 
to give an upper bound for the superconcentration constant, in the following form:

\begin{corollary} \label{cor:bd-s-R}
Let $f$ be a smooth function on $\mathbb R^n$. For the following parameters: 
\begin{align} \label{eq:R-f}
s(f) :=\frac{{\rm Var}[f(G)]}{\mathbb E \|\nabla f(G)\|_2^2} , \quad 
R(f) := \frac{ \mathbb E\|\nabla f(G)\|_2^2}{ \sum_{j=1}^n \|\partial_j f\|_{L^1}^2 }, 
\end{align} we have
\begin{align} \label{eq:bd-s-R-2}
\frac{2}{s(f)} \gr \alpha \log R(f),
\end{align}  where $\alpha>0$ is a universal constant.
\end{corollary}

This estimate follows from Talagrand's inequality with an appropriate application of Jensen's inequality 
(see e.g. \cite[Chapter 5]{Cha} for the details).
Since an explicit constant is required for our approach (see Remark \ref{rem:opt-Tal}), 
we provide a proof at the end of \S \ref{sec:s-sb} which yields $\alpha=1$. The proof is nothing
more than a repetition of the existing argument in \cite[Chapter 5]{Cha}, or \cite{BKS}, 
with a careful bookkeeping of the constants in each computational stage.

Let us point out that Talagrand \cite{Tal-russo} proved \eqref{eq:Tal-L1-L2} for the uniform probability measure on the Hamming cube. 
An alternative approach was presented by Benjamini, Kalai and 
Schramm in \cite{BKS}. Both approaches rest on the Bonami-Beckner hypercontractive inequality and therefore \eqref{eq:Tal-L1-L2} holds true for any
hypercontractive measure, see e.g. \cite{CorLed}. An explicit proof of the Gaussian version (Theorem \ref{thm:Tal-L1-L2}) can be found in \cite{CorLed} or 
\cite[Chapter 5]{Cha}. Our approach also depends on the hypercontractive property of the Gaussian measure. 
We recall some basic facts in the next paragraph.

\subsection{Ornstein-Uhlenbeck semigroup}

The hypercontractive property of the Gaussian measure can be expressed in terms of the associated 
Ornstein-Uhlenbeck semigroup (OU-semigroup). 
Let us first recall the definition. For any $f\in L^1(\gamma_n)$ we define 
\begin{align} \label{eq:Mehler}
P_tf(x) = \mathbb Ef \left(e^{-t}x+\sqrt{1-e^{-2t}} G\right ), \quad x\in \mathbb R^n, \quad t\gr 0,
\end{align} where $G$ is the standard Gaussian vector in $\mathbb R^n$. It is known that $P_t f$ is solution of the following heat equation:
\begin{align} \label{eq:heat-eqt}
	\left\{ \begin{array}{cc}
			\partial_t u = {\mathcal L} u \\ 
			u(x,0)=f 
				\end{array} \right. ,
\end{align} where $\mathcal L$ stands for the generator of the semigroup, that is ${\mathcal L} u = \Delta u - \langle x, \nabla u \rangle$ for
sufficiently smooth functions $u$. For background material on the OU-semigroup we refer the reader to \cite{BGL}.
In the next lemma we collect several properties of the semigroup $(P_t)_{t\gr 0}$ that will be useful in our approach.

\begin{lemma} \label{lem:ou-1} The Onrnstein-Uhlenbeck semigroup $(P_t)_{t\gr 0}$ enjoys the following properties:
\begin{enumerate}
\item $P_0f(x)=f(x)$ and $P_\infty f(x) =\lim_{t\to \infty} P_tf(x) =\mathbb E[f(G)]$, $x\in \mathbb R^n$.
\item If $f$ is smooth, then for each $i\ls n$ we have $\partial_i (P_tf)= e^{-t} P_t(\partial_i f)$ .
\item $\mathbb E P_tf(G) = \mathbb Ef(G)$.
\item If $f\gr 0$, then $P_tf\gr 0$.
\item For every $1\ls p\ls \infty$ we have that $P_t:L^p \to L^p$ is a linear contraction.
\item If $f$ is convex, then $P_tf$ is also convex.
\end{enumerate}
\end{lemma}

These properties can be easily verified from the definition of the OU-semigroup. Alternatively, proofs of these facts can be 
found, e.g., in \cite[Chapter 5]{Led-book} or \cite{BGL}. Let us emphasize that while the properties (1)-(5) are satisfied
by a fairly general class of Markov semigroups (see e.g. \cite{BGL}), property (6) depends crucially on the fact 
that the OU-semigroup admits the integral 
representation \eqref{eq:Mehler} with respect to the Mehler kernel. Indeed; for any $\lambda\in [0,1]$ and 
$x,y\in \mathbb R^n$ we may write
\begin{align*}
P_tf( (1-\lambda)x +\lambda y) &= \mathbb E \left[ f \left( e^{-t} \left( (1-\lambda)x +\lambda y \right) +\sqrt{1-e^{-2t}} G \right) \right] \\
& =  \mathbb E \left[ f \left( (1-\lambda) \left(e^{-t}x + \sqrt{1-e^{-2t}}G \right) + \lambda \left (e^{-t} y +\sqrt{1-e^{-2t}} G \right) \right) \right] \\
& \ls \mathbb E \left[ (1-\lambda) f \left( e^{-t}x + \sqrt{1-e^{-2t}} G \right) + \lambda  f \left (e^{-t} y +\sqrt{1-e^{-2t}} G \right) \right] \\
&= (1-\lambda) P_tf(x) + \lambda P_tf(y),
\end{align*} where we have used the convexity of $f$ pointwise.

The fact that $P_t$ is a linear contraction admits an improvement by, roughly
speaking, relaxing the integrability assumption in the domain.
This important property of the Ornstein-Uhlenbeck semigroup called hypercontractivity is due to Nelson \cite{Nel}. We have the 
following:

\begin{theorem} [Nelson] \label{thm:ou-hyper}
Let $f:\mathbb R^n\to \mathbb R$ be a function and let $1<p \ls q$. Then, for any $t\gr 0$ with $q \ls 1+ e^{2t}(p-1)$, we have
\begin{align}
\|P_tf\|_{L^q(\gamma_n)} \ls \|f\|_{L^p(\gamma_n)}.
\end{align}
\end{theorem}

Gross \cite{Gro} showed that the hypercontractive property is intimately connected to the logarithmic Sobolev inequality. 
In particular, he proved that a probability measure is hypercontractive if and only if satisfies a logarithmic Sobolev inequality. 
See also \cite[Chapter 5]{BGL}.

\subsection{Smoothening convex functions}

We are now ready to present the proofs of the new small ball estimates 
announced in the Introduction (Theorem \ref{thm:main-1}) in terms of the parameters $\beta$ and $\tilde \beta$.

\subsubsection{From hypercontractivity to small ball estimates}\label{subsec:hsb}

Recall that for any smooth function $f$ with $\mathbb E[f(G)] \neq 0$ we define 
\begin{align}
\tilde \beta(f) = \frac{\mathbb E\|\nabla f(G)\|_2^2}{(\mathbb E[f(G)] )^2}.
\end{align}

Our goal is to prove the following small ball estimate:

\begin{theorem} \label{thm:hyper-sb}
Let $f:\mathbb R^n \to [0,\infty)$ be a positively homogeneous,
convex map with $\|\nabla f\|_2 \in L^2(\gamma_n)$ and 
let $\tilde \beta =\tilde \beta(f)$. 
Suppose that for some 
$L>0$ we have
\begin{align}\label{eq:eq-L1-bal}
\sum_{j=1}^n \|\partial_j f\|_{L^1(\gamma_n)}^2 \ls L (\mathbb Ef(G))^2. 
\end{align} Then, for any $\delta \in (0,1/2)$ we have
\begin{align*}
\mathbb P \left\{ f(G) \ls \delta \mathbb E[f(G)] \right\} \ls 
\exp\left\{ -c\delta^2 \left( \frac{1}{\tilde \beta}\right )^{\tau(\delta)} \left( \frac{1}{L}\right)^{1-\tau(\delta)} \right\},
\end{align*} where $\tau(\delta) \asymp \delta^2$.
\end{theorem}

Next proposition is the key observation.

\begin{proposition} \label{prop:key-1} Let $f$ be a function with $\|\nabla f\|_2 \in L^2$ and let $L>0$ be such that
\begin{align*}  
\sum_{j=1}^n    \|\partial_j f\|_{L^1(\gamma_n)}^2 \ls L (\mathbb E[f(G)] )^2.
\end{align*} Then, for all $t\gr 0$ we have
\begin{align*}
\frac{1}{\tilde \beta(P_tf)} \gr e^{2t}  \left( \frac{1}{\tilde \beta(f) } \right)^{ \frac{2e^{- 2t}} {1+e^{-2t} } }\left( \frac{1}{L} \right)^{ \frac{1-e^{- 2t}}{1+e^{-2t}} }.
\end{align*} 
\end{proposition}

For the proof we will need the following lemma which is a consequence of Nelson's hypercontractivity. 

\begin{lemma} \label{lem:A-1}
Let $h$ be a function in $L^2$. Then, for all $t\gr 0$ one has
\begin{align}
\|P_t h\|_{L^2} \ls \|h\|_{L^2} \left(\frac{\|h\|_{L^1}}{\|h\|_{L^2}} \right)^{\tanh t}.
\end{align}
\end{lemma}

\noindent {\it Proof.} Using the hypercontractivity (Theorem \ref{thm:ou-hyper}) we get
\begin{align*}
\|P_t h\|_2 \ls \|h\|_p, \quad p=1+e^{-2t}.
\end{align*} By H\"older's inequality we derive
\begin{align*}
\|h\|_p^p \ls \|h\|_1^{2-p} \|h\|_2^{2(p-1)}.
\end{align*} The result follows. \prend

\bigskip

Now we turn to proving the key proposition. The argument we follow can be traced back to 
\cite{BKS} (see also \cite[Theorem 5.1]{Cha}):

\bigskip

\noindent {\it Proof of Proposition \ref{prop:key-1}.} Let $A_i= \|\partial_i f\|_{L^2}$ and $a_i=\|\partial_if\|_{L^1}$. We may write
\begin{align} \label{eq:key-1-1}
\mathbb E \| \nabla (P_tf)\|_2^2 &= e^{-2t} \sum_{i=1}^n \| P_t(\partial_i f) \|_{L^2}^2 
\ls e^{-2t} \sum_{i=1}^n A_i^2 \left( \frac{a_i^2}{A_i^2}\right)^{\tanh t },
\end{align} where we have also used Lemma \ref{lem:A-1}. Note that the function $u\mapsto u^{\tanh t }, \; u>0$ is concave, thus 
Jensen's inequality implies
\begin{align} \label{eq:key-1-2}
\sum_{i=1}^n A_i^2 \left( \frac{a_i^2}{A_i^2}\right)^{\tanh t } \ls \mathbb E\|\nabla f\|_2^2 \left( \frac{\sum_{i=1}^n a_i^2}{\mathbb E\|\nabla f\|_2^2}\right)^{\tanh t}.
\end{align} Combining \eqref{eq:key-1-1} with \eqref{eq:key-1-2} we arrive at
\begin{align} \label{eq:ineq-grad}
\mathbb E \| \nabla (P_tf)\|_2^2 \ls  e^{-2t} \mathbb E \| \nabla f \|_2^2 R(f)^{-\tanh t},
\end{align} where $R(f)$ was defined in \eqref{eq:R-f}.
Dividing both sides with $(\mathbb E[f(G)])^2$, using the assumption and the definition of $\tilde \beta(f)$, the above estimate yields
\begin{align*}
\tilde \beta(P_tf) \ls e^{-2t} \tilde \beta(f) \left( \frac{L}{\tilde \beta(f)} \right)^{ \tanh t},
\end{align*}
as required. \prend

\bigskip

Now we turn to proving our first main result in this section.

\bigskip

\noindent {\it Proof of Theorem \ref{thm:hyper-sb}.}  Since $f$ is positively homogeneous and convex
it is also sub-additive, hence
\begin{align*}
P_tf(x) \ls e^{-t} f(x) +\sqrt{1-e^{-2t}}\mathbb Ef(G), \quad x\in \mathbb R^n, \quad t\gr 0.
\end{align*} We fix $\delta\in (0,1/2)$. For any $t>0$, using Theorem \ref{thm:sdh}, we may write
\begin{align*}
\mathbb P \left( f(G) \ls \delta \mathbb E f(G) \right) 
\ls \mathbb P \left( P_tf(G) \ls \left( \delta e^{-t}+\sqrt{1-e^{-2t}} \right) \mathbb E [P_tf(G)] \right) \ls \exp \left( -c \varepsilon(t)^2 / \tilde\beta(P_tf) \right),
\end{align*} provided that 
\begin{align*}
\varepsilon(t):=1-\delta e^{-t} -\sqrt{1-e^{-2t}} = e^{-t} \left( \frac{e^{-t}}{1+\sqrt{1-e^{-2t}}} -\delta \right) >0.
\end{align*} We select time $t_0=t_0(\delta)>0$ such that 
\begin{align*}
\frac{e^{-t_0}}{1+\sqrt{1-e^{-2t_0}}}= 2\delta.
\end{align*} For this choice we have $ e^{-t_0} \asymp \delta $ and by taking into account Proposition \ref{prop:key-1} we conclude 
with 
\[\tau(\delta) = \frac{2e^{-2t_0}}{1+e^{-2t_0}} \asymp \delta^2 .\] 
The proof is complete. \prend

\subsubsection{From superconcentration to small ball estimates} \label{sec:s-sb}

Recall the definition of the parameter $\beta$. For any smooth function $f$ with $\mathbb Ef\neq 0$ we have
\begin{align}
\beta(f) = \frac{{\rm Var}[f(G)]}{(\mathbb E[f(G)] )^2}.
\end{align}

The main result of this section is the following variant of Theorem \ref{thm:hyper-sb}:

\begin{theorem} \label{thm:super-sb}
Let $f:\mathbb R^n\to [0,\infty)$ be a positively homogeneous, convex function in $L^2(\gamma_n)$ 
and let $\beta=\beta(f)$. Suppose that 
\begin{align*}
\sum_{j=1}^n \| \partial_j f\|_{L^1}^2 \ls L (\mathbb Ef(G))^2, \quad L>0.
\end{align*} Then, for any $\delta \in (0, 1/2)$ we have
\begin{align*}
\mathbb P \left\{ f(G) \ls \delta \mathbb Ef(G) \right\} \ls 
\exp \left\{ - c \delta^2 \left(\frac{1}{\beta} \right)^{\omega(\delta) } \left( \frac{1}{L} \right)^{1-\omega(\delta)} \right\},
\end{align*} with $\omega(\delta) \asymp \delta^\alpha$ and $\alpha$ is the constant from Corollary \ref{cor:bd-s-R}.
\end{theorem}

The proof follows the same lines as in the previous paragraph with the appropriate modifications. Unlike to the previous one, 
this approach exploits the exponential decay of the variance along the flow $P_t$ as was previously discussed in the Introduction. 
More precisely, for any 
Markov semigroup $P_t$ with invariant measure $\mu$ one has 
\begin{align} \label{eq:var-decay}
{\rm Var}(P_tf) \ls e^{-2\lambda_1 t} {\rm Var}(f), \quad t\gr 0,
\end{align} where $\lambda_1$ is the spectral gap of $\mu$. This well known property can be found in any classical text concerned with
semigroup tools, see e.g. \cite{Led-surv, BGL}. However, this dimension free estimate cannot later provide dimension dependent bounds of the
right order of magnitude, which is the central theme of this work due to the conjectured extreme case of $\ell_\infty$-norm. Note that the latter
is genuinely high-dimensional functional whereas \eqref{eq:var-decay} takes into account the extremals in Poincar\'e inequality, i.e. linear functionals.
In other words, the bound in terms of $\lambda_1$ is the worst case scenario amongst the ones we have to encounter in the light of
\[
\lambda_1(\mu)= \inf \left\{ \frac{1}{s_\mu(f)} : f \; {\rm smooth}, \; {\rm nonconstant} \right\}.
\]
By carefully revisiting the proof of \eqref{eq:var-decay} we see that something more is true, which fixes this sub-optimality issue. That said,
the proof of the next lemma is standard and can be found implicitly in \cite{Led-surv, BGL}. Since the formulation we need 
is not explicitly stated there, we shall include a (sketch of) proof for reader's convenience. 

\begin{lemma} \label{lem:ou-2}
Let $f$ be in $L^2(\gamma_n)$. We define $v(t)={\rm Var}[P_tf(G)]$ and $s(t) = s(P_tf)$.
Then, we have the following properties:

\begin{itemize}
\item [i.] $v'(t)=-2\mathbb E\|\nabla (P_tf) \|_2^2$.

\item [ii.] $v'(t) s(t)= -2v(t)$, hence for all $t\gr0$ we have
\begin{align*}
v(t)=v(0) \exp \left( -2 \int_0^t \frac{dz}{s(z)} \right).
\end{align*}
\item [iii.] $v(t)$ is log-convex and hence, $s(t)$ is nondecreasing.
\end{itemize}

\end{lemma}
\begin{proof} ({\it Sketch}).
Since all assertions include shift invariant quantities we may assume without loss of generality that $\mathbb Ef =0$. Thus, 
$v(t) =\mathbb E (P_tf)^2$ and differentiation in terms of $t$, under the integral sign, yields
\[
\frac{dv}{dt} = \frac{d}{dt} \left( \mathbb E (P_tf)^2 \right) = 2\mathbb E[P_tf {\mathcal L}P_t f] = -2 \mathbb E \langle \nabla P_tf,  \nabla P_tf \rangle
=-2\mathbb E \|\nabla (P_tf)\|_2^2,
\] where we have used that $P_t$ solves the heat equation \eqref{eq:heat-eqt} and the generator $\mathcal L$ satisfies the integration by parts formula,
i.e. $\mathbb E (u{\mathcal L}v) =- \mathbb E \langle \nabla u, \nabla v \rangle$ for $u,v$ smooth functions.
The second assertion readily follows. For the log-convexity of $t \mapsto v(t)$ note that 
\[
\frac{d^2 v}{dt^2} = \frac{d}{dt} \left( 2 \mathbb E [P_{2t}f {\mathcal L}f] \right) = 4 \mathbb E[{\mathcal L}P_{2t}f {\mathcal L}f] = 4 \mathbb E[({\mathcal L}P_tf)^2],
\] where we have used the fact that $P_t$ and ${\mathcal L}$ commute and that $P_t$ is self-adjoint. The third assertion now 
follows by the Cauchy-Schwarz inequality. For more details we refer the reader to \cite[p.183-184]{BGL}.
\end{proof}

\begin{remark} The above lemma provides a link between the superconcentration phenomenon 
and the variance decay, during the ensuing motion. 
This can be viewed as an alternative definition to the superconcentration. 
Alternative (equivalent) definitions (via the gap in Poincar\'e inequality or in connection with chaos) can be found in \cite[Chapter 3 \& 4]{Cha}. 
\end{remark}

We will need the next proposition, which is the analogue of Proposition \ref{prop:key-1} in terms of the parameter $\beta(f)$.
For the proof we employ Corollary \ref{cor:bd-s-R} which we take for granted until we prove it (with $\alpha=1$) at the end
of this section.

\begin{proposition} \label{prop:key} Let $f\in L^2(\gamma_n)$ be a smooth function and let $L>0$ such that
\begin{align*}  
\sum_{j=1}^n  \| \partial_j f\|_{L^1} ^2 \ls L (\mathbb Ef(G))^2.
\end{align*} Then, for all $t\gr 0$ we have
\begin{align*}
\frac{1}{\beta(P_tf)} \gr  e^{2t-2}\left( \frac{1}{\beta (f)} \right)^{ e^{- \alpha t}}\left( \frac{1}{L} \right)^{1-e^{- \alpha t} },
\end{align*} where $\alpha$ is the constant from Corollary \eqref{cor:bd-s-R} 
\end{proposition}

\noindent {\it Proof.} To ease the notation we set $\beta(t)=\beta(P_tf)$ and $\beta=\beta(f)$. First, note that Lemma \ref{lem:ou-2} yields
\begin{align} \label{eq:key-1}
\frac{1}{\beta(t)} =\frac{1}{\beta} \exp(\psi(t)), \quad \psi(t):= 2 \int_0^t \frac{dz}{s(z)}.
\end{align} Next, we employ Corollary \ref{cor:bd-s-R} to link $s(t)$ with $\beta(t)$. To this end, it suffices to bound from below the 
parameter $R(t)=R(P_tf)$. Indeed; we have 
\begin{align*}
R(t) = \frac{\mathbb E\|\nabla (P_tf) (G)\|_2^2}{ \sum_j \|\partial_j (P_tf)\|_{L^1}^2 } \gr
\frac{\mathbb E \|\nabla (P_tf)(G)\|_2^2}{e^{-2t} \sum_j \|\partial_j f\|_{L^1}^2} \gr
\frac{ e^{2t} {\rm Var}[P_tf (G)]}{  L(\mathbb Ef (G))^2} = e^{2t} \beta(t)/L,
\end{align*} where we have also used Lemma \ref{lem:ou-1} and the Poincar\'{e} inequality \eqref{eq:Poin}. Whence, we obtain
\begin{align} \label{eq:key-2}
\psi'(t) = \frac{2}{s(t)} \gr  \alpha \log( e^{2t} \beta(t)/L).
\end{align}
Inserting \eqref{eq:key-1} into \eqref{eq:key-2} we obtain
\begin{align*}
\psi' (t) \gr \alpha \log\left( \frac{\beta}{L} e^{2t-\psi(t) }\right) \Longrightarrow \psi'(t) + \alpha \psi(t) \gr 2 \alpha t+ \alpha \log(\beta/L). 
\end{align*} Integrating the above in $[0,t]$ we find
\begin{align*}
e^{\alpha t} \psi(t) = \int_0^t e^{\alpha z} \left( \psi'(z) + \alpha \psi(z) \right) \, dz &\gr \int_0^t e^{\alpha z} \left( 2\alpha z+\alpha \log(\beta/L) \right) \, dz \\
& = 2te^{\alpha t} - \frac{2}{\alpha}(e^{\alpha t}-1) + (e^{\alpha t} -1) \log ( \beta /L) \\
&= 2te^{\alpha t} + (e^{\alpha t} -1) \log ( e^{-2/\alpha} \beta /L).
\end{align*} It follows that
\begin{align*}
\psi(t) \gr 2t + (1-e^{-\alpha t}) \log(e^{-2/\alpha} \beta /L), \quad t\gr 0. 
\end{align*}
Plug the latter in \eqref{eq:key-1} we arrive at
\begin{align*}
\frac{1}{\beta(t)} \gr \frac{e^{2t}}{\beta} \left( \frac{\beta}{ e^{2/\alpha} L} \right)^{1-e^{-\alpha t}}  \; \Longrightarrow  \;
\frac{1}{\beta(t)} \gr e^{2t} \left( \frac{1}{\beta} \right)^{e^{-\alpha t}} \left( \frac{\beta}{ e^{2/\alpha} L} \right)^{1-e^{-\alpha t}} ,
\end{align*}
as required. \prend

\medskip

Now we turn to proving our second main result.

\medskip

\noindent {\it Proof of Theorem \ref{thm:super-sb}.} The proof follows the same lines as before: For $\delta\in (0,1/2)$ and $t>0$ we have
\begin{align*}
\mathbb P \left\{ f(G) \ls \delta \mathbb E f(G) \right\} \ls \exp \left( -c \varepsilon(t)^2 /\beta(t) \right),
\end{align*} provided that 
\begin{align*}
\varepsilon(t):=1-\delta e^{-t} -\sqrt{1-e^{-2t}} = e^{-t} \left( \frac{e^{-t}}{1+\sqrt{1-e^{-2t}}} -\delta \right) >0.
\end{align*} Selecting time $t_0=t_0(\delta)>0$ such that 
\begin{align*}
\frac{e^{-t_0}}{1+\sqrt{1-e^{-2t_0}}}= 2\delta,
\end{align*} and taking into account Proposition \ref{prop:key}, we conclude 
with $\omega(\delta) = \exp(-\alpha t_0) \asymp \delta^\alpha$. \prend

\begin{remark} \label{rem:opt-Tal}
	The above argument demonstrates that the optimal constant $\alpha_{\rm opt}$ in inequality \eqref{eq:bd-s-R-2} 
	is directly
	connected with the dependence of $\delta$ in the small ball estimate established in Theorem \ref{thm:super-sb}. 
	We show below that $\alpha_{\rm opt}\gr 1$ and an a posteriori examination shows that $\alpha_{\rm opt}\ls 2$, at least
	in the case of even functions. If happens $\alpha_{\rm opt}=2$, then Theorem \ref{thm:hyper-sb} is of course redundant. However, 
	we are not aware of the value $\alpha_{\rm opt}$ as of this writing. Until this interesting question is clarified the two results 
	are incomparable as we have already explained in the Introduction.  
\end{remark}

\medskip

\noindent {\it Proof of Corollary \ref{cor:bd-s-R} } (with $\alpha=1$). Recall estimate \eqref{eq:ineq-grad},
\begin{align*} 
\mathbb E\|\nabla (P_tf)\|_2^2 \ls e^{-2t} \mathbb E\|\nabla f\|_2^2  R(f)^{- \tanh t}, \quad t\gr 0.
\end{align*} Integrating the latter and taking into account the fact that
\begin{align*}
{\rm Var}[f(G)]= v(0)= - \int_0^\infty v'(t) \, dt = 2 \int_0^\infty   \mathbb E\|\nabla (P_tf)\|_2^2 \, dt ,
\end{align*} which follows from Lemma \ref{lem:ou-1}, we get
\[
s(f) \ls 2\int_0^\infty e^{-2t} R(f)^{-\tanh t} \, dt  = 2 \int_0^\infty e^{-2t} e^{-x \tanh t} \, dt,
\] where $x:=\log R(f) \gr 0$. The following easy fact from calculus completes the proof:

\smallskip

\noindent {\it Fact.} For $x>0$ one has
\[
J(x) := \int_0^\infty e^{-2t} e^{-x \tanh t}\, dt < \frac{1}{x}, \qquad J(x)\sim 1/x, \quad x\to \infty.
\] 

\smallskip

\noindent {\it Proof of Fact.} Apply the change of variable $z= \tanh t$ to get
\[
J(x) = \int_0^1 \frac{1}{(1+z)^2} e^{-xz} \, dz .
\]
Integration by parts yields
\[
xJ(x) = 1 -\frac{e^{-x}}{4} - 2\int_0^1 \frac{e^{-xz}}{(1+z)^3} \, dz < 1.
\] Moreover, $xJ(x) \sim 1$ as $x\to \infty$ which shows that the universal constant $\alpha=1$ we obtain, 
in terms of the integral $J(x)$, is asymptotically optimal. \prend

\bigskip


\section{Applications to asymptotic geometric analysis}

In this section we apply Theorem~\ref{thm:hyper-sb} to derive optimal small ball estimates 
in normed spaces in terms of the underlying dimension. 
Since our study takes into account the unconditional structure of the norm, both explicitly and implicitly, 
we begin with some auxiliary results in this context.
Our approach blends with analytic techniques hence, in several instances, the norm under study is required 
to be smooth enough. In order to ease the exposition we assume throughout the section, without loss of generality, 
that the norms are sufficiently smooth; the general case follows by a standard approximation argument, 
as described e.g., in \cite{PV-dicho}. Let us point out that in the case of unconditional norms the smooth approximation
can be arranged in order to preserve the unconditionality. 

\subsection{Unconditional structure}

Let $\|\cdot\|$ be a norm on $\mathbb R^n$ and let $(b_i)$ be a basis. Following \cite{FJ}, 
we define the unconditional constant of the norm with respect to the basis, denoted by  ${\rm unc}(\|\cdot\|, \{b_i\})$, to be the least $r>0$ such that
\begin{align}
\left\| \sum_{i=1}^n \varepsilon_i \alpha_i b_i\right\| \ls r \left\| \sum_{i=1}^n \alpha_i b_i \right\| ,
\end{align} for all choices of signs $\varepsilon_i = \pm 1$ and all scalars $(\alpha_i) \subset \mathbb R$. Note that 
${\rm unc}(\|\cdot\|, \{b_i\}) \gr 1$ for any basis $\{b_i\}$. Next, one defines 
\begin{align}
{\rm unc}(\mathbb R^n, \|\cdot\|) := \inf \left\{ {\rm unc}(\|\cdot\|, \{b_i\}) : (b_i) \; \textrm{basis} \right\}.
\end{align}
We denote by $(e_i)$ the standard (orthonormal) basis in $\mathbb R^n$.

\medskip

 The following lemma can be viewed as a Lozanovski type result (see e.g. \cite{Sz-volr} or \cite{Pis-book}):
 
 \begin{lemma} \label{lem:prop-bal-pos} 
 Let $\|\cdot\|$ be a norm on $\mathbb R^n$, let $a_i= \mathbb E| \partial_i \|G\| |$ for $i=1,\ldots,n$ and let $r={\rm unc}(\|\cdot\|, \{e_i\})$. 
 \begin{enumerate}
 \item For every $x\in \mathbb R^n$, we have
 \begin{align} \label{eq:bal-pos-1}
\frac{1}{r} \sum_{i=1}^n a_i |x_i| \ls \|x\| \ls r \sqrt{ \frac{\pi}{2} } \mathbb E\|G\| \cdot \|x\|_\infty.
 \end{align}
 \item The following estimate holds:
 \begin{align} \label{eq:bal-pos-2}
 c \frac{ \mathbb E\|G\|}{\sqrt{\log n} } \ls \sum_{i=1}^n a_i \ls r \sqrt{\frac{\pi}{2}} \mathbb E\|G\|.
 \end{align} In particular, 
 \begin{align*}
 \frac{c'}{n \log n} \ls \frac{\sum_{i=1}^n a_i^2}{(\mathbb E\|G\|)^2} .
 \end{align*}
 \item Assuming that $a_i=a_j$ for all $i,j=1,2,\ldots,n$,  we have the following  estimate:
 \begin{align} \label{31-part3}
 {\rm vol}_n(B_X)^{1/n} \mathbb E\|G\| \ls C r \sqrt{\log n},
 \end{align} where $B_X = \{x\in \mathbb R^n : \|x\| \ls 1 \}$.
 \end{enumerate}
 \end{lemma}

\noindent {\it Proof.} (1). Note that for all $x\in\mathbb R^n$ and for (almost) every $y$ we have
\[
\sum_{i=1}^n x_i \cdot \partial_i\|y\| =\langle x, \nabla \|y\| \rangle  \ls \|x\|,  
\] where we used the fact that $\nabla \|y\|$ belongs to the unit sphere of the dual space (whenever is defined).
Applying the latter for $x_i\to x_i {\rm sgn} \{ x_i \partial_i\|y\| \}  \equiv \varepsilon_i x_i$, we obtain 
\begin{align*}
\sum_{i=1}^n |x_i| \cdot | \partial_i\|y\| | \ls \left\| \sum_{i=1}^n \varepsilon_i x_i e_i\right\| \ls r \|x\|,
\end{align*} for all $x$ and for (almost) every $y$. 
Integration over $y$ with respect to the Gaussian measure yields the lower estimate. For the upper estimate we may argue as follows:
\begin{align*}
\|x\| \ls r \mathbb E_\delta \left\| \sum_{i=1}^n \delta_i x_ie_i \right \| \ls r \|x\|_\infty \mathbb E_\delta \left\| \sum_{i=1}^n \delta_i e_i \right\|,
\end{align*} where $(\delta_i)$ are independent Rademacher and we have used the contraction principle \cite[Theorem 4.4]{LedTal-book}. 
Jensen's inequality and the fact that $(\delta_i |g_i|)$ have
the same distribution as $(g_i)$ (see e.g. \cite[Proposition 1]{Pis} or \cite[Lemma 4.5]{LedTal-book}) yields
\[ \mathbb E \left\| \sum_{i=1}^n \delta_i e_i \right\| \ls ( \mathbb E |g_1| )^{-1} \mathbb E \left\| \sum_{i=1}^n g_i e_i \right\|. \] The assertion follows.

\smallskip

\noindent (2). The rightmost inequality follows by integrating the left-hand side of \eqref{eq:bal-pos-1} with respect to $x$ and the Gaussian measure: 
If $G=(g_1, \ldots, g_n)$ is a standard Gaussian vector, then 
\[
\mathbb E|g_1| \cdot \sum_{i=1}^na_i = \mathbb E \left ( \sum_i^n a_i |g_i|  \right) \ls r\mathbb E\|G\|.
\] It remains to notice that $\mathbb E|g_1|=\sqrt{2/\pi}$. 

For the leftmost we argue as follows: Set $A=\{\|G\|_\infty> C \sqrt{\log n}\}$ and note that $\mathbb P(A)\ls  n^{-6}$ for a 
sufficiently large absolute constant $C>0$. Then, we may write
\begin{align*}
\mathbb E\|G\| = \mathbb E \left[  \sum_{i=1}^n g_i \partial_i \|G\| \right] \ls \mathbb E \left[ \|G\|_\infty \cdot \left\| \nabla \|G\| \right\|_1 \right]
\ls C\sqrt{\log n} \, \mathbb E \left \| \nabla \|G\| \right\|_1 + n \sqrt{\pi /2} \cdot \mathbb E\|G\| \cdot  \mathbb E[\|G\|_\infty \mathbf 1_A] ,
\end{align*} where we have also used the fact that
\footnote{Let $\|u\|_2=1$ such that $\max_{\|\theta\|_2=1}\|\theta\|=\|u\|$. By duality there exists
$v$ such that $\langle x,v\rangle \ls \|x\|$ for all $x$ and $\langle u, v \rangle =\|u\|$. 
On the other hand by the rotation invariance of the Gaussian measure we get 
$\mathbb E\|G\| \gr \mathbb E | \langle G, v \rangle | = \sqrt{2/ \pi} \|v\|_2 \gr \sqrt{2/ \pi} \|u\|$, as claimed.}

\[ 
|\partial_i \|G\| | \ls \| \nabla \|G\| \|_2 \ls \max_{\|\theta\|_2=1} \|\theta\| 
\ls \sqrt{\pi/2} \, \mathbb E\|G\|, \quad a.s. 
\] It remains to notice that 
\[
\mathbb E [ \|G\|_\infty \mathbf 1_A] \ls \sum_{i=1}^n \mathbb E[|g_i|\mathbf 1_A] \ls n \sqrt{\mathbb P(A)},
\] where in the last passage we have applied the Cauchy-Schwarz inequality and the standard fact that $\mathbb E g_i^2=1$.
Finally, the in particular part follows from the leftmost estimate and the Cauchy-Schwarz inequality.

\smallskip

\noindent (3). The left-hand side of \eqref{eq:bal-pos-1} implies that $B_X \subseteq \frac{r}{a} B_1^n$. 
Hence, taking volumes on both sides, we get
\begin{align*}
{\rm vol}_n(B_X)^{1/n} \ls \frac{r}{a} \frac{2}{(n!)^{1/n}} \ls \frac{2er}{an} \ls \frac{2er}{c \mathbb E\|G\|} \sqrt{\log n},
\end{align*} where in the last step we have used the leftmost estimate from \eqref{eq:bal-pos-2}. \prend

\begin{remark} \label{rem:3-1}
\begin{enumerate}
	\item Both lower and upper estimates in \eqref{eq:bal-pos-2} are sharp (up to constants) in the case of $\ell_\infty^n$ and $\ell_1^n$-norm, respectively. 
	
	\item Note that H\"older's inequality implies that ${\rm vol}( B_X )^{1/n} \mathbb E\|G\| \gr c$. Indeed; we may write
	\[
	\mathbb E\|G\| = \mathbb E\|G\|_2 \int_{S^{n-1}} \|\theta\| \, d\sigma(\theta),
	\] where $\sigma(\cdot)$ is the uniform probability measure on the unit Euclidean sphere $S^{n-1}=\{\theta \in \mathbb R^n: \|\theta\|_2=1\}$.
	By H\"older's inequality we get 
	\[
	\int_{S^{n-1}} \|\theta\| \, d\sigma(\theta) \gr \left( \int_{S^{n-1}} \|\theta\|^{-n} \, d\sigma(\theta) \right)^{-1/n}.
	\] On the other hand, integration in polar coordinates yields ${\rm vol}_n(B_X) = n \omega_n \int_{S^{n-1}} \|\theta \|^{-n} \, d\sigma(\theta),$ 
	 where $\omega_n$ is the volume of the Euclidean ball $B_2^n = \{x \in \mathbb R^n: \|x\|_2\ls 1\}$.
	 It remains to notice that $\mathbb E\|G\|_2 \asymp \omega_n^{1/n} \asymp \sqrt{n}$. 
	 Thus, \eqref{31-part3} should be viewed as a reverse H\"older estimate.
	
	\item In the sequel only the estimate \eqref{eq:bal-pos-1} and \eqref{eq:bal-pos-2} will be used. 
	However, we state the lemma in this form for the sake of completeness and for future reference.
\end{enumerate}
\end{remark}

\begin{theorem} \label{thm:opt-unc}
Let $\|\cdot\|$ be a norm in $\mathbb R^n$ which satisfies $\mathbb E|\partial_i\|G\| | = \mathbb E|\partial_j\|G\| |$ 
for all $i,j=1,2,\ldots,n$. If $r={\rm unc}( \|\cdot \|, \{e_i\})$, then for any $\delta\in (0,1/2)$ we have
\begin{align}
\mathbb P \left\{ \|G\| \ls \delta \mathbb E\|G\| \right\} < 
\exp \left( - c \left( \frac{1}{\tilde \beta} \right)^{\tau(\delta)} \left( \frac{n}{r^2} \right)^{1-\tau(\delta)}  \right),
\end{align} with $\tau(\delta) \asymp \delta^2$, where $\tilde \beta = \tilde \beta( \|\cdot\|)$.
\end{theorem}

\noindent {\it Proof.} By Theorem \ref{thm:hyper-sb} we get
\begin{align*}
\mathbb P \left\{ \|G\| \ls \delta \mathbb E\|G\| \right\} < \exp \left( - c \delta^2 (1/\tilde \beta)^{\tau(\delta)} (1/L)^{1-\tau(\delta)} \right),
\end{align*} where $L$ is given by
\begin{align*}
L = \frac{\sum_{i=1}^n (\mathbb E | \partial_i \|G\| |)^2 }{(\mathbb E\|G\|)^2}.
\end{align*}
We distinguish two cases.

\begin{itemize}
\item $\delta \gr \frac{1}{10 r} \sqrt{Ln}$. Then, $1/L \gr \frac{n}{100 r^2 \delta^2}$ and thus,
\begin{align*}
\delta^2 (1/L)^{1-\tau(\delta)} \gr c' (n/r^2)^{1-\tau(\delta) }. 
\end{align*}

\item $\delta< \frac{1}{10 r} \sqrt{Ln}$. Note that by Lemma \ref{lem:prop-bal-pos} we have
\begin{align*}
a \|x\|_1 \ls r \|x\|, \quad a=\mathbb E| \partial_i \|G\| | = \mathbb E\|G\| \sqrt{\frac{L}{n}},
\end{align*} 
by the definition of $L$ and the equality for the partial derivatives of the 
norm. Therefore, we obtain
\begin{align*}
\{x: \|x\| \ls \delta \mathbb E\|G\|\} \subset \left \{ x: \|x\|_1 \ls \frac{\delta r}{a} \mathbb E\|G\|  \right\} 
\subset \left\{x: \|x\|_1 \ls \frac{n}{10}  \right\},
\end{align*} where for the last inclusion we have used the assumption on $\delta$. This inclusion yields
$\mathbb P\{ \|G\| \ls \delta \mathbb E\|G\| \} < e^{-cn}$. Note that $r\gr 1$ and 
in view of \eqref{eq:beta-comp} and \eqref{eq:beta-magni} it is also $\tilde\beta \gtrsim 1/n$, hence we obtain 
\[
\left(\frac{1}{\tilde \beta} \right)^{\tau(\delta)} \left( \frac{n}{r^2}\right)^{1-\tau(\delta)} \ls ( C_1n)^{\tau(\delta)} n^{1-\tau(\delta)} \ls C_2n,
\] as required.
\end{itemize} 
In each case we get the desired result. \prend

\medskip

The following is immediate corollary of Theorem \ref{thm:opt-unc}:

\begin{theorem}\label{thm:1-unc}
Let $\|\cdot\|$ be a 1-unconditional norm on $\mathbb R^n$, which satisfies $\mathbb E|\partial_i\|G\| | = \mathbb E|\partial_j\|G\| |$ 
for all $i,j=1,2,\ldots,n$. Then, for any $\delta\in (0,1/2)$ we have
\begin{align}
\mathbb P \left\{ \|G\| \ls \delta \mathbb E\|G\| \right\} < \exp \left( - c n^{1-C\delta^2}  \right),
\end{align} where $c,C>0$ are universal constants.
\end{theorem}

\noindent {\it Proof.} It suffices to notice that $\tilde \beta \lesssim 1$ in view of \eqref{eq:beta-comp} and \eqref{eq:beta-magni}. \prend

\subsection{$1$--unconditional norms in the $\ell$-position and position of minimal $M$}\label{subs: lm position}

Theorem \ref{thm:hyper-sb} allowed to prove optimal small-ball estimates for $1$--unconditional norms
in the position where absolute first gaussian moments of partial derivatives of the norm are all equal (Theorem~\ref{thm:1-unc}).

The ``$\ell$-position" is a classical position in local theory of Banach spaces which played a crucial role for the development 
of the theory, see e.g. \cite{Pis-book}. This position, as in the case of $w^{1,1}$-position admits a description via an isotropic condition, 
namely a norm $\|\cdot\|$ on $\mathbb R^n$ is in $\ell$-position if 
\[
\mathbb E (G_i  \, \partial_i\|G\| \, \|G\| )= \frac{\mathbb E [\|G\|^2]}{n}, \quad i=1,2,\ldots, n.
\]
A closely related position is the ``position of minimal $M$". Similarly, we have that a norm $\|\cdot\|$ on $\mathbb R^n$
is in minimal $M$-position if satisfies
\begin{align} \label{eq:M-iso}
\mathbb E (G_i \, \partial_i \|G\| ) = \frac{\mathbb E\|G\|}{n}, \quad i=1,2,\ldots,n.
\end{align} For further details and how one obtains the aforementioned isotropic characterization we refer the reader to \cite{GM}.

In the $\ell$ or minimal $M$ positions, a direct adaptation of the argument
will lead to a weaker estimate for small $\delta$
``close'' to $\log^{-1/2} n$. The reason is that partial derivatives of the norm in the $\ell$-- (or in minimal $M$--) position
are not necessarily equal and can vary up to the
factor $O(\log^{1/2} n)$. 
Indeed, in that setting we rely on the identities $\Exp(G_i\,\partial_i\|G\|)=\Exp(G_j\,\partial_j\|G\|)$
for all $i\neq j$ (in the position of minimal $M$) and $\Exp(G_i\,\partial_i\|G\|\,\|G\|)=\Exp(G_j\,\partial_j\|G\|\,\|G\|)$
(in the $\ell$--position) which imply $\Exp|\partial_i\|G\||\ls C\sqrt{\log n}\,\Exp|\partial_j\|G\||$
for all $i\neq j$. 
\footnote{In view of the unconditionality we have $x_i \partial_i \|x\| = |x_i \partial_i \|x\| |$ for all
$i=1,2,\ldots, n$ and for all $x$ and $|\partial_i \|x\| |$ is increasing function of $|x_i|$. Hence, by Chebyshev's 
association inequality 
we get
\[
\mathbb E[G_i \partial_i \|G\|] = \mathbb E \left[ |G_i| \cdot | \partial_i \|G\| |  \right] \gr \mathbb E|G_i| \cdot \mathbb E| \partial_i \|G\||.
\] On the other hand, an argument similar to the one for proving \eqref{eq:bal-pos-2} shows that 
\[
\mathbb E|G_i \partial_i \|G\|| \ls C \sqrt{\log n} \, \mathbb E |\partial_i \|G\| | + \frac{c}{n^2} \mathbb E\|G\|.
\] The claim follows if we take into account the balancing condition \eqref{eq:M-iso}.
}
Thus, considering a direct sum of $\ell_2$ and $\ell_\infty$ (of appropriate dimensions),
one can check that the inequalities are in general optimal up to the constant multiple $C$ (see also Remark \ref{rem:3-1}). 
We leave the details to the interested reader.

In this subsection, we prove the small-ball inequality for $1$--unconditional norms
in position of minimal $M$, which matches the estimate in Theorem~\ref{thm:1-unc}.
For this, we augment the above argument, based on Theorem~\ref{thm:hyper-sb},
with specially constructed norm replacement.
The proof can be repeated for the $\ell$--position with minor modifications,
so we will provide the details only for the former.

\begin{theorem}\label{p: optimal in M}
Let $\|\cdot\|$ be a $1$--unconditional norm in the position of minimal $M$.
Then for any $\delta\in(0,1/2]$ we have
$$\Prob\big\{\|G\|\ls \delta\Exp\|G\|\big\}\ls \exp(-c n^{1-C\delta^2}),$$
where $c,C$ are universal constants.
\end{theorem}

The main technical step of the proof is to construct (for every fixed $\delta$) a seminorm $\mathcal T={\mathcal T}_\delta$ satisfying three
conditions: 
\begin{itemize}
\item $\Exp\mathcal T(G)\gr c\,\Exp\|G\|$, 
\item $\mathcal T(x)\ls C\|x\|$ for all $x\in\R^n$, and
\item 
$\frac{\sum_{i=1}^n (\Exp | \partial_i \mathcal T(G) |)^2 }{(\Exp \mathcal T(G))^2}
\ls C\delta^2/n$. 
\end{itemize}
The first two conditions immediately yield
$$\Prob\big\{\|G\|\ls \delta\Exp\|G\|\big\}\ls \Prob\big\{\mathcal T(G)\ls C'\delta\,\Exp\mathcal T(G)\big\},$$
while the third condition, together with Theorem~\ref{thm:hyper-sb}, implies the desired
deviation estimate. The seminorm is constructed in Proposition~\ref{p: d2} as a composition
of three mappings: a diagonal contraction $\widetilde D$, an auxiliary seminorm $\mathcal U(\cdot)$
and a mapping $F_{\cdot,\cdot}(\cdot)$ defined below.
Each of the mappings from the composition is responsible for particular structural properties of $\mathcal T(\cdot)$:
$\widetilde D$ is needed to ``balance'' the quantities $\Exp G_i\partial_i\mathcal T(G)$,
the auxiliary seminorm $\mathcal U(\cdot)$ provides an upper bound on the $\ell_1^n$--norm of norming functionals,
and $F_{\cdot,\cdot}(\cdot)$ controls the size of the support of the norming functionals,
which, together with the estimate for $\Exp G_i\partial_i\mathcal T(G)$,
implies the upper bound for $\sum_{i=1}^n (\Exp | \partial_i \mathcal T(G) |)^2$.

\medskip

Everywhere in this subsection, by a {\it norming functional} for a seminorm $\|\cdot\|$ in $\R^n$
we mean any vector $x\in\R^n$ such that $\sup\limits_{\|y\|\ls 1}\langle y,x\rangle=1$
and $\langle y,x\rangle=0$ for all $y$ with $\|y\|=0$.

As the first step of the proof, we define a mapping $F$ on the class of all $1$--unconditional seminorms in $\R^n$.
\begin{definition}
Let $\|\cdot\|$ be a $1$--unconditional seminorm, and fix a parameter $\tau\gr 1$.
We will define a new $1$--unconditional seminorm $F_{\|\cdot\|,\tau}(\cdot)$ as follows.
Take any norming functional $x$ for the seminorm $\|\cdot\|$.
We associate with $x$ a collection of vectors $\{v(x,I)\}_{I}\subset\R^n$ indexed over all
subsets $I\subset[n]$, and with each $v(x,I)$ defined by
$$v(x,I):=\Big(1+\frac{\|x\,{\bf 1}_I\|_{\ell_1^n}}{\tau\|x\|_{\ell_1^n}}\Big)x\,{\bf 1}_{[n]\setminus I},$$
where ${\bf 1}_{J}$ denotes the indicator of a subset of indices $J$. 
In a sense, we truncate and rescale the original functional $x$. 
Now, set
$$F_{\|\cdot\|,\tau}(y):=\sup\limits_{x}\sup\limits_{I\subset[n]}\langle v(x,I),y\rangle,\quad y\in\R^n,$$
where the supremum is taken over all norming functionals $x$ for the seminorm $\|\cdot\|$.
\end{definition}
It is immediately clear that $F_{\|\cdot\|,\tau}(\cdot)$ is a $1$--unconditional seminorm.
Another elementary observation is

\begin{lemma}\label{l: norm trick 1}
For any $1$--unconditional seminorm $\|\cdot\|$ and $\tau\gr 1$, the seminorm $F_{\|\cdot\|,\tau}(\cdot)$ satisfies
$$\|y\|\ls F_{\|\cdot\|,\tau}(y)\ls \Big(1+\frac{1}{\tau}\Big)\|y\|,\quad \forall\,y\in\R^n.$$
\end{lemma}

\noindent {\it Proof.} For the lower estimate note that $\{v(x,\emptyset)\}$ 
is the collection of the original norming functionals. On the other hand, the unconditionality implies coordinatewise
monotonicity, hence $\| x\mathbf 1_{[n]\setminus I}\|_\ast \ls \|x\|_\ast$ for any $I\subset [n]$ and $x$, where
$\|\cdot\|_\ast$ stands for the dual norm of $\|\cdot\|$. \prend

The following is a crucial property that says that norming functionals for $F_{\|\cdot\|,\tau}(\cdot)$
are supported on the sets of large coordinates of corresponding vectors.
\begin{lemma}\label{l: F transform}
Let $K>0$ be a parameter, and let
$\|\cdot\|$ be a $1$--unconditional seminorm such that every norming functional for $\|\cdot\|$ has $\ell_1^n$--norm
at most $K$. Further, let $\tau\gr 1$ and let $F_{\|\cdot\|,\tau}(\cdot)$ be as above.
Take a vector $y\in\R^n$ with $\|y\|\neq 0$ and let $\widetilde x$ be a norming functional with respect to the
seminorm $F_{\|\cdot\|,\tau}(\cdot)$, such that $\langle y,\widetilde x\rangle=F_{\|\cdot\|,\tau}(y)$. Then necessarily
$$\supp\,\widetilde x\subset\Big\{i\ls n:\;|y_i|\gr \frac{\tau\|y\|}{(\tau+1)^2 K}\Big\}.$$ 
\end{lemma}
\begin{proof}
Let $x$ be the ``original'' norming functional with respect to $\widetilde x$,
i.e.\ let $x$ be a norming functional with respect to the seminorm $\|\cdot\|$,
and let $I\subset\supp(\widetilde x)^c$ be the set such that
$$\widetilde x=\Big(1+\frac{\|x\,{\bf 1}_I\|_{\ell_1^n}}{\tau\|x\|_{\ell_1^n}}\Big)x\,{\bf 1}_{[n]\setminus I}.$$
Since our seminorms are $1$--unconditional, we may assume without loss of generality that all components
of $y$ and $x$ are non-negative.
Fix any $i\in \supp\widetilde x$, define $J:=I\cup\{i\}$, and set
$$x':=\Big(1+\frac{\|x\,{\bf 1}_J\|_{\ell_1^n}}{\tau\|x\|_{\ell_1^n}}\Big)x\,{\bf 1}_{[n]\setminus J}.$$
Then $x'$ belongs to the collection of functionals $\{v(x,I)\}_I$ from the definition of $F_{\|\cdot\|,\tau}(\cdot)$.
Hence, $\langle \widetilde x,y\rangle=F_{\|\cdot\|,\tau}(y)\gr\langle x',y\rangle$.
On the other hand,
$$\frac{1}{\widetilde \alpha}\langle \widetilde x,y\rangle
=\frac{1}{\alpha'}\langle x',y\rangle+x_i y_i,
$$
with
$$\widetilde \alpha:=\Big(1+\frac{\|x\,{\bf 1}_I\|_{\ell_1^n}}{\tau\|x\|_{\ell_1^n}}\Big);\quad
\alpha':=\Big(1+\frac{\|x\,{\bf 1}_J\|_{\ell_1^n}}{\tau\|x\|_{\ell_1^n}}\Big).$$
Note that
$$\frac{\alpha'}{\widetilde\alpha}=\frac{\tau\|x\|_{\ell_1^n}+\|x\,{\bf 1}_J\|_{\ell_1^n}}{\tau\|x\|_{\ell_1^n}
+\|x\,{\bf 1}_I\|_{\ell_1^n}}
=\frac{\tau\|x\|_{\ell_1^n}+\|x\,{\bf 1}_I\|_{\ell_1^n}+x_i}{\tau\|x\|_{\ell_1^n}+\|x\,{\bf 1}_I\|_{\ell_1^n}}.$$
Therefore, from the above we get
$$\langle \widetilde x,y\rangle\gr
\frac{\tau\|x\|_{\ell_1^n}+\|x\,{\bf 1}_I\|_{\ell_1^n}+x_i}{\tau\|x\|_{\ell_1^n}
+\|x\,{\bf 1}_I\|_{\ell_1^n}}\langle \widetilde x,y\rangle
-\alpha' x_i y_i,\quad \mbox{ so that }\quad\frac{\langle \widetilde x,y\rangle}{\tau\|x\|_{\ell_1^n}
+\|x\,{\bf 1}_I\|_{\ell_1^n}}
\ls\alpha' y_i.$$
This inequality, together with the relation $\|y\|\ls F_{\|\cdot\|,\tau}(y)=\langle \widetilde x,y\rangle$
from Lemma~\ref{l: norm trick 1}, implies that
$$
\frac{\|y\|}{(\tau+1)\|x\|_{\ell_1^n}}\ls
\frac{\|y\|}{\tau\|x\|_{\ell_1^n}+\|x\,{\bf 1}_I\|_{\ell_1^n}}\ls
\frac{\langle \widetilde x,y\rangle}{\tau\|x\|_{\ell_1^n}+\|x\,{\bf 1}_I\|_{\ell_1^n}}
\ls \alpha'y_i\ls \Big(1+\frac{1}{\tau}\Big)y_i,$$
and the result follows.
\end{proof}

\begin{proposition}\label{p: d1}
Let $\theta\in (0,1)$ and let $n\gr C/\theta^2$. Let $\|\cdot\|$ be a $1$--unconditional norm in the position of minimal $M$, with $\Exp\|G\|=1$,
and let $\delta>0$. Then at least one of the following two assertions is true:
\begin{itemize}

\item Either $\Prob\{\|G\|\ls \delta\Exp\|G\|\}\ls \exp(-c \theta^2 n)$, or 

\item There is a $1$--unconditional seminorm ${\mathcal U}={\mathcal U}_\delta$ in $\R^n$ such that
${\mathcal U}(y)\ls\|y\|$ for all $y\in\R^n$;
all norming functionals for ${\mathcal U}(\cdot)$ have $\ell_1^n$--norm at most $C \delta /\theta^3$; and for any subset $J\subset[n]$
of size $|J|\ls c\theta^2n$ we have $\Exp{\mathcal U}(G\,{\bf 1}_{[n]\setminus J})\gr (1-\theta)\Exp\|G\|$. 
\end{itemize}
\end{proposition}
\begin{proof}
Let $I$ be the subset of all indices $i$ such that $\Exp|\partial_i\|G\||\gr \rho \delta/n$,
where $\rho=\rho(\theta)$ will be chosen later.
Note that for any vector $y\in\R^n$ we have
\begin{equation}\label{eq: aux 072450298}
\|y\,{\bf 1}_I\|\gr \sum_{i\in I}|y_i| \cdot \mathbb E|\partial_i\|G\| | \gr \frac{\rho \delta}{n}\|y\,{\bf 1}_{I}\|_1,
\end{equation} where we have also used Lemma \ref{lem:prop-bal-pos}. We will consider two possibilities:
\begin{itemize}

\item $|I|\gr \frac{2n}{\rho}$. Then, in view of \eqref{eq: aux 072450298} and applying standard concentration estimates
for Lipschitz functions, we get
\begin{align*}
\Prob\big\{\|G\|\ls \delta\Exp\|G\|\big\}&\ls \Prob\big\{\|G\,{\bf 1}_I\|\ls \delta\big\}\\
&\ls \Prob\big\{\|G\,{\bf 1}_I\|_1\ls n/\rho \big\}\\
&\ls \Prob\big\{\|G\,{\bf 1}_I\|_1\ls |I|/2\big\}\ls \exp(-c n/\rho),
\end{align*}
for some universal constant $c>0$.

\item $|I|< \frac{2n}{\rho}$.
First, let
\[
B^\ast := \left \{ y \in \mathbb R^n: \langle x, y \rangle \ls 1, \; \forall x\in \mathbb R^n, \; \|x\| \ls 1 \right \}.
\]
Note that for any norm $\|\cdot\|$ we have that $\nabla \|x\|\in B^\ast$ (in view of $\|x\| = \langle x, \nabla\|x\| \rangle$) whenever
$x$ is a point of differentiability for $\|\cdot\|$.
For any $t\gr 1$ define a collection of functionals
$$S(t):=\big\{x\,{\bf 1}_{[n]\setminus I}:\;x\in B^*\;\;\mbox{ and }\;\;\|x\,{\bf 1}_{[n]\setminus I}\|_{1}\ls \rho \delta t
\big\}.$$
By our assumption and by Markov's inequality, we have
$$\Prob\Big\{\sum\limits_{i\in [n]\setminus I}|\partial_i\|G\||\ls \rho\delta t\Big\}\gr 1-\frac{1}{t},\quad t\gr 1.$$
Choose any subset $J\subset[n]$. By Markov's inequality again 
and the assumption that $\|\cdot\|$ is in the position of minimal $M$ (see \eqref{eq:M-iso}), we have
$$\Prob\Big\{\sum\limits_{i\in I\cup J}G_i\,\partial_i\|G\|\ls t(|I|+|J|)/n\Big\}\gr 1-\frac{1}{t},\quad t\gr 1.$$
Since $\sum_{i\in [n]\setminus (I\cup J)} x_i \partial_i \|x\| = \|x\| - \sum_{i\in I\cup J} x_i \partial_i \|x\|$ for a.e. $x$, we obtain
\begin{align*}
\Prob \left\{ \sum_{i\in [n]\setminus (I\cup J)}G_i\,\partial_i\|G\|\gr \left(\|G\|-\frac{t (|I|+|J|)}{n} \right)
\, {\rm and} \,
\sum_{i\in [n]\setminus I}|\partial_i\|G\||\ls \rho \delta\,t
\right \} \gr 1-\frac{2}{t},
\end{align*} for all $t\gr 1$.
Together with the definition of $S(t)$ this gives
$$
\Prob\Big\{\sup\limits_{w\in S(t)}\langle G\,{\bf 1}_{[n]\setminus J},w\rangle\gr \big(\|G\|-t(|I|+|J|)/n\big)
\Big\}\gr 1-\frac{2}{t},\quad t\gr 1,\quad J\subset[n].
$$
Finally, we define
${\mathcal U}(\cdot)$ as
$${\mathcal U}(y):=\sup\big\{\langle y,w\rangle:\;w\in S(8/\theta)\big\},\quad y\in\R^n.$$
Clearly, ${\mathcal U}(\cdot)$ is a $1$--unconditional seminorm with ${\mathcal U}(\cdot)\ls \|\cdot\|$.
Further, it is known that in position of minimal $M$ one has that the Dvoretzky number $ k( \| \cdot \|)$ (see \eqref{eq:conc-dvo} for the definition)
is at least of order 
$ \log{n}$ \cite[Proposition 2.5]{PV-dicho}. So, since we have assumed that $ \mathbb E \| G \|=1 $ we have that  (for large enough $n$)  
$\|G\|\gr 1-\frac{\theta}{8}$ with probability at least $1-\frac{\theta}{8}$.
Then the above relations imply that for any $J\subset[n]$ with $|J|\ls 2n/\rho$ we have
$$
\Prob\Big\{{\mathcal U}( G\,{\bf 1}_{[n]\setminus J})\gr (1-\/8-32/(\rho \theta))
\Big\}\gr 1-\frac{\theta}{8}-\frac{\theta}{4}.
$$
Finally, choose $\rho:=\frac{512}{\theta^2}$.
It is then easy to see that the last relation implies that
$\Exp{\mathcal U}(G\,{\bf 1}_{[n]\setminus J})\gr 1-\theta$,
completing the proof.
\end{itemize}
\end{proof}


The next technical lemma describes a continuous diagonal contraction,
which will be used later to ``balance'' coordinates of a seminorm.

\begin{lemma}\label{l: aux 09209234}
Let $\mathcal V(\cdot)$ be any $1$--unconditional seminorm in $\R^n$.
Denote by ${\bf D}$ the set of diagonal $n\times n$ matrices with diagonal entries in $[0,1]$.
Then for any $L>0$ there is a differentiable function $D:[0,\infty)\to{\bf D}$ having the following properties:
\begin{itemize}

\item $D(t)$ is a solution to initial value problem
$$
\begin{cases}
D(0):=I_n;\\
\frac{\partial}{\partial t}d_{ii}(t)=H(D(t))_i,
\;\;i\ls n,\quad t\in[0,\infty),
\end{cases}
$$
where for each $i\ls n$, and $A=(a_{kj})\in{\bf D}$
$$
H(A)_i:=-\max\big(0,\Exp (G_i\,\partial_i ({\mathcal V}\circ A)(G))-L\big)
-\max\Big(0,\frac{1}{4}-\Big| a_{ii}-\frac{1}{4}\Big|\Big),
$$
and $\partial_i ({\mathcal V}\circ A)(G)$ denotes $i$-th partial derivative of the seminorm
$$y\longmapsto {\mathcal V}\Big(\sum\limits_{i=1}^n a_{ii}y_ie_i\Big) = {\mathcal V}(Ay), \quad y\in\R^n,$$
at point $G$;

\item Setting $\widetilde D=(\widetilde d_{ii})\in{\bf D}$ to be the entry-wise limit of $D(t)$ when $t\to\infty$,
we have $\widetilde d_{ii}\in\{0\}\cup[1/2,1]$ for all $i\ls n$,
$$\Exp (G_i\,\partial_i ({\mathcal V}\circ \widetilde D)(G))\ls L ,$$
and 
$$\big|\big\{i\ls n:\;\widetilde d_{ii}=0\big\}\big|\ls \frac{2\Exp{\mathcal V}(G)}{ L}.$$

\end{itemize} 
\end{lemma}
\begin{proof}
For any $A\in{\bf D}$ we have $H(A)\in(-\infty,0]^n$,
and the function $H$ is Lipschitz continuous everywhere on ${\bf D}$. Indeed; using Gaussian integration by parts we may write
\[
H(A)_i = - \max \left(0,  \mathbb E[(G_i^2-1) {\mathcal V}\circ A (G)] - L \right) - \max  \left( 0, \frac{1}{4} - \left| a_{ii}-\frac{1}{4}\right| \right).
\] Since the mapping $A\mapsto \mathbb E [ (G_i^2-1){\mathcal V}\circ A(G) ]$ is Lipschitz continuous for $i=1,2,\ldots,n$, so is the mapping 
$A\mapsto H(A)$ on $\bf D$. 
Hence, the initial value problem stated above has a (global) solution, by the Picard theorem. Moreover, each solution $d_{ii}$ is 
a nonincreasing function with $0\ls d_{ii} \ls 1$, since $\frac{\partial}{\partial t} d_{ii}(t) \ls 0$ and
$D(t)\in \bf D$. Hence, $\widetilde d_{ii} = \lim_{t\to \infty} d_{ii}(t)$ exists in $[0,1]$.
We also have the following:

\smallskip

\noindent {\it Fact.} For the solution $d_{ii}$ we have that $\frac{\partial}{\partial t} d_{ii}$ is uniformly continuous, hence
$\lim_{t\to \infty} \frac{\partial}{\partial t}d_{ii}(t)=0$. 

\smallskip

\noindent {\it Proof of Fact.} 
Note that each solution $d_{ii}$ satisfies $| \frac{\partial}{\partial t} d_{ii}(t)| = |H(D(t))_i| \ls 1 + L + \mathbb E [ |G_i^2-1| {\mathcal V}(G) ] = : K_i$, 
where we have used the $1$-unconditionality of $\mathcal V$ and the fact that $0\ls d_{ii}(t) \ls 1$ for all $t\gr 0$. Hence, $d_{ii}$ is
$K_i$-Lipschitz. Furthermore, we may write 
\[
\left| \frac{\partial d_{ii}}{\partial z} (t) -   \frac{\partial d_{ii}}{\partial z} (s)\right| = |H(D(t))_i - H(D(s))_i| \ls {\rm Lip}(H) \cdot \|D(t)-D(s)\|_{\rm op} 
\ls \max_i K_i \cdot {\rm Lip}(H) \cdot |t-s|,
\] which proves that $\partial d_{ii} / \partial t$ is uniformly continuous. Since $\lim_{t \to \infty} d_{ii}(t)$ exists, we infer
that $\lim_{t\to \infty} \frac{\partial}{\partial t}d_{ii}(t)=0$ (Barbalat's lemma). \prend

Having proved that fact, we may argue as follows: 
By taking limits on both sides of the differential equation 
(where we have first switched $\Exp (G_i\,\partial_i ({\mathcal V} \circ \widetilde D)(G))$ to $\mathbb E [(G_i^2-1) {\mathcal V}(\widetilde D G)]$ 
in order to interchange limit with expectation) 
and using the fact we find 
$$\max \left( 0, \Exp [(G_i^2-1) {\mathcal V}( \widetilde DG) ]-L \right) = \max \left( 0, \frac{1}{4} - \left| \widetilde d_{ii} - \frac{1}{4} \right| \right)=0, \quad i\ls n,$$
which implies $\widetilde d_{ii} \in \{0\} \cup [1/2,1]$ and $\Exp (G_i\,\partial_i ({\mathcal V} \circ \widetilde D)(G))\ls L$, as claimed.
\footnote{However, this bound is meaningful
for all $i\ls n$ with $\widetilde d_{ii} \neq 0$, since for the others holds trivially $\Exp (G_i\,\partial_i ({\mathcal V} \circ \widetilde D)(G))=0<L$.
}

Continue with the proof, for the next property we may write
\begin{align*}
\mathbb E[{\mathcal V}(G)] &= \mathbb E [{\mathcal V}(\widetilde DG)] - \int_0^\infty \mathbb E \left[ \langle \nabla {\mathcal V} (D(t)G) , D'(t) G \rangle\right] \, dt \\
& = \mathbb E [{\mathcal V}(\widetilde DG)] + \sum_{i=1}^n \int_0^\infty \left(- \frac{\partial}{\partial t} d_{ii}(t) \right) \mathbb E \left[ G_i  {\partial_i \mathcal V} (D(t)G) \right] \, dt \\
& \gr \mathbb E [{\mathcal V}(\widetilde DG)] + \sum_{i=1}^n \int_0^\infty \mathbf 1_{\{d_{ii}(t) \gr 1/2\} }\left(- \frac{\partial}{\partial t} d_{ii}(t) \right)  \mathbb E \left[ G_i {\partial_i \mathcal V} (D(t)G) \right] \, dt\\
& = \mathbb E [{\mathcal V}(\widetilde DG)] + \sum_{i=1}^n \int_0^\infty \mathbf 1_{\{d_{ii}(t) \gr 1/2\} }\frac{ -\frac{\partial }{\partial t}d_{ii}(t) }{d_{ii} (t) }  
\mathbb E \left[ G_i \partial_i (\mathcal V \circ D(t)) (G) \right] \, dt,
\end{align*} where we have used that $\partial d_{ii} / \partial t\ls 0$ and 
that $x_i \partial_i {\mathcal V}(x) = | x_i \partial_i {\mathcal V}(x)| \gr 0$ for almost all $x$ 
due to the $1$-unconditionality of $\mathcal V$. 

Now, our definition of $D(t)$ implies that $\frac{\partial}{\partial t}d_{ii}(t)<0$ and $d_{ii}(t)\gr 1/2$ hold simultaneously
only if $\Exp (G_i\,\partial_i ({\mathcal V} \circ D(t))(G))\gr L$.
Using this inequality and the fact that $d_{ii}\ls 1$, we get
$$
\Exp{\mathcal V}(G)\gr \Exp{\mathcal V}(\widetilde DG)+
\sum\limits_{i=1}^n \int\limits_{0}^\infty (-L){\bf 1}_{\{d_{ii}(t)\gr 1/2\}}
\frac{\partial}{\partial t}d_{ii}(t)\,dt\gr \Exp{\mathcal V}(\widetilde DG)+ \frac{ L}{2}\big|\big\{i\ls n:\;
\widetilde d_{ii}=0\big\}\big|,
$$
implying the bound on the cardinality of $\big|\big\{i\ls n:\;
\widetilde d_{ii}=0\big\}\big|$.

\end{proof}

\begin{remark}
On the conceptual level, the continous contraction $D(t)$ constructed above is designed to
act on the coordinates which give the main input to the expectation $\Exp {\mathcal V}(G)$.
This way, we balance coordinates by making their input approximately equal.
When ``reasonable'' balancing does not work for some coordinates, that is, when after being rescaled by $1/2$
they still produce a large input to the norm, we zero them out.
\end{remark}

\begin{definition}
Let $\|\cdot\|$ be a $1$--unconditional seminorm in $\R^n$ and let $\tau\gr 1$ be a parameter.
We will say that $\|\cdot\|$
is in {\it $M_\tau$--position} if for any $i\ls n$ we have
$$\Exp (G_i\,\partial_i\|G\|)\ls \frac{\tau}{n}\Exp\|G\|,$$
where $G$ is the standard Gaussian vector in $\R^n$.
\end{definition}

\begin{proposition}\label{p: d2}
There are universal constants $C,c>0$ with the following property.
Let $n\gr C$, let $\|\cdot\|$ be a $1$--unconditional norm in the position of minimal $M$, with $\Exp\|G\|=1$,
and let $\delta>0$. Then at least one of the following is true:
\begin{itemize}

\item Either $\Prob\{\|G\|\ls \delta\Exp\|G\|\}\ls \exp(-c n)$, or 

\item There is a $1$--unconditional seminorm ${\mathcal T}={\mathcal T}_{\delta}$ in $\R^n$ in the $M_C$--position such that
${\mathcal T}(y)\ls 2\|y\|$ for all $y\in\R^n$;
$\Exp {\mathcal T}(G)\gr \frac{1}{4}\Exp\|G\|$; and
$\Exp|\partial_i {\mathcal T}(G)|\ls \frac{C\delta}{n}\Exp\|G\|$ for all $i\ls n$.
\end{itemize}
\end{proposition}
\begin{proof}
We will apply Proposition~\ref{p: d1} with $\theta:=1/2$.
Assuming that the constant $c>0$ is sufficiently small and that the first assertion of
the statement does not hold, there is a $1$--unconditional seminorm ${\mathcal U}(\cdot)$
with the properties stated in Proposition~\ref{p: d1}.
Define a seminorm ${\mathcal W}(\cdot)$ in $\R^n$ by setting
$${\mathcal W}(y):=F_{{\mathcal U}(\cdot),1}(y),\quad y\in\R^n,$$
where the transformation $F_{{\mathcal U}(\cdot),1}$ was constructed earlier.
Note that, in view of Lemma~\ref{l: norm trick 1} and properties of ${\mathcal U}(\cdot)$,
\begin{equation}\label{eq: aux 09482069235}
\Exp{\mathcal W}(G)\ls 2.
\end{equation}

Next, we apply Lemma~\ref{l: aux 09209234} with $\mathcal V:=\mathcal W$ and $L:=\frac{\tau}{n}\Exp\|G\|$
($\tau$ to be chosen a bit later)
to obtain a diagonal contraction operator $\widetilde D$
with $\widetilde d_{ii}\in\{0\}\cup[1/2,1]$, $i\ls n$, and with the set
$$J:=\big\{i\ls n:\;\widetilde d_{ii}=0\big\}$$
of cardinality at most $\frac{2n\Exp {\mathcal W}(G)}{\tau}$,
which is less than $\frac{4n}{\tau}$, in view of \eqref{eq: aux 09482069235}.
Set
$${\mathcal T}(y):={\mathcal W}(\widetilde Dy),\quad y\in\R^n.$$
We claim that ${\mathcal T}(\cdot)$ is a $1$--unconditional seminorm in the $M_{C}$--position,
for an appropriate $C>0$.
By the properties of $\widetilde D$, for all $i\in[n]\setminus J$ we have $\widetilde d_{ii}\gr 1/2$,
whence
${\mathcal T}(y)\gr \frac{1}{2}{\mathcal W}(y)$ for any $y\in\R^n$ with $\supp y\subset [n]\setminus J$.
Take $\tau:=4\theta^{-2}/c$, where $c$ is the constant from Proposiion~\ref{p: d1}.
Then, by that proposition, we have
$$\Exp {\mathcal T}(G)\gr \Exp {\mathcal T}(G\,{\bf 1}_{[n]\setminus J})
\gr \frac{1}{2}\Exp {\mathcal W}(G\,{\bf 1}_{[n]\setminus J})\gr \frac{1}{4}\Exp \|G\|.$$
Hence, by Lemma~\ref{l: aux 09209234} and the definition of $L$,
$$\Exp (G_i\partial_i {\mathcal T}(G))\ls \frac{\tau}{n}\Exp\|G\|\ls \frac{4\tau}{n}\Exp {\mathcal T}(G),\quad i\ls n,$$
so indeed $\mathcal T(\cdot)$ is in $M_C$--position for $C:=4\tau$.

It only remains to check the assertion about the partial derivatives, i.e.\ that
$$\Exp|\partial_i {\mathcal T}(G)|\ls \frac{C''\delta}{n}\Exp\|G\|,\quad i\ls n,$$
for some $C''>0$.
For that, we will apply Lemma~\ref{l: F transform}.
Pick any vector $y\in\R^n$ with $\supp y\subset [n]\setminus J$ and ${\mathcal T}(y)\neq 0$.
Observe that $\widetilde Dx$ is a norming functional for $y$ with respect to the seminorm ${\mathcal T}(\cdot)$
if and only if $x$ is a norming functional for $\widetilde Dy$ with respect to the seminorm ${\mathcal W}(\cdot)$. 
On the other hand, according to Lemma~\ref{l: F transform}, we have
$$\supp\,x\subset\Big\{i\ls n:\;|\widetilde d_{ii}y_i|\gr \frac{{\mathcal U}(y)}{2^2 K}\Big\},$$
where $K$ is the maximal $\ell_1^n$--norm of a norming functional for $\mathcal U(\cdot)$.
The definition of $\mathcal U(\cdot)$ implies that $K\ls C'\delta$. 
Since all diagonal entries of $\widetilde D$ indexed over $[n]\setminus J$, belong to the interval $[1/2,1]$, we get from the above
$$\supp\,\widetilde Dx\subset\Big\{i\ls n:\;|y_i|\gr \frac{{\mathcal T}(y)}{16 C'\delta}\Big\},$$
or, equivalently,
$$\partial_i {\mathcal T}(z)\neq 0\;\;\mbox{ only if }|z_i|\gr \frac{{\mathcal T}(z)}{16 C'\delta}\;\;
\mbox{ for almost all $z\in\R^n$}.$$ 
But this immediately implies that for any $i\ls n$ we have
$$|\partial_i {\mathcal T}(G)|\ls \frac{16 C'\delta}{\mathcal T(G)}G_i\partial_i {\mathcal T}(G)$$
almost everywhere on the probability space, whence, together with the above relations,
$$\Exp|\partial_i {\mathcal T}(G)|\ls \frac{C_1\tau\delta}{n}$$
for a universal constant $C_1>0$. The result follows.
\end{proof}

\begin{proof}[Proof of Theorem~\ref{p: optimal in M}]
We need to show that
\begin{align}
\Prob \left\{ \|G\| \ls \delta \Exp\|G\| \right\} <
2\exp \big( - c n^{1-C\delta^2}  \big),\quad \delta\in(0,1/2].
\end{align}
The proof essentially follows by combining Proposition~\ref{p: d2} and Theorem~\ref{thm:hyper-sb}.
If the first assertion of Proposition~\ref{p: d2} holds then we are done.
Otherwise, let $\mathcal T(\cdot)$ be the seminorm defined within the second assertion of that proposition.
We set
\begin{align*}
L := \frac{\sum_{i=1}^n (\Exp | \partial_i \mathcal T(G) |)^2 }{(\Exp \mathcal T(G))^2}.
\end{align*}
Then, by the properties of $\mathcal T(\cdot)$, we have $L\ls C\delta^2/n$ for some $C>0$.
On the other hand, again by the properties of $\mathcal T(\cdot)$,
$$\Prob \left\{ \|G\| \ls \delta \Exp\|G\| \right\}
\ls \Prob \left\{\mathcal T(G) \ls 8\delta \Exp\mathcal T(G) \right\}.
$$
It remains to apply Theorem~\ref{thm:hyper-sb}.
\end{proof}

\subsection{The general case}\label{subs: gen case}

Here we show that every norm in $\mathbb R^n$ has an invertible linear image
which satisfies a small-ball estimate with exponential decay in $n$, thus extending Theorem \ref{thm:1-unc} and
Theorem \ref{p: optimal in M} to arbitrary norms. More precisely, we prove the following:

\begin{theorem} \label{thm:main-2a}
Let $\|\cdot\|$ be a norm on $\mathbb R^n$. Then, there exists $T\in GL(n)$ with the following property: for any $0<\delta < \frac{1}{4}$ one has
\begin{align}
\mathbb P \left\{ \|TG\| \ls \delta \mathbb E\| TG\| \right\} \ls \exp \left( - c  n^{\frac{1}{4}-C\delta^2} \right), 
\quad G\sim N({\bf 0},I_n),
\end{align} where $c , C>0$ are universal constants.
\end{theorem}

The strategy is first to show that a power law estimate holds true conditionally, i.e. if the underlying norm
has moderate unconditional structure. In view of Theorem~\ref{thm:opt-unc}, this can be accomplished if the norm also satisfies the 
$w^{1,1}$--condition. This is promised by the following lemma from \cite{PV-dicho} whose proof rests on the Borsuk-Ulam
antipodal theorem.

\begin{lemma} \label{lem:bal-der}
Let $f$ be a smooth norm on $\mathbb R^m$. Then, there exist $\lambda_1,\ldots, \lambda_m>0$ such that 
\begin{align*}
\|\partial_i (f\circ \Lambda)\|_{L^1} = \| \partial_j (f\circ \Lambda) \|_{L^1}, \quad i,j=1,\ldots,m,
\end{align*} where $\Lambda={\rm diag} (\lambda_1, \ldots,\lambda_m)$. 
\end{lemma}

Combining this with Theorem~\ref{thm:opt-unc} we derive the following:

\begin{theorem} \label{thm:opt-r-unc}
Let $\|\cdot\|$ be a norm on $\mathbb R^m$. There exists  $S\in GL(m)$ such that, for any $\delta\in (0,1/2)$ we have
\begin{align}
\mathbb P \left\{ \|SG\| \ls \delta \mathbb E\|SG\| \right\} < \exp \left( - c \left( m/ r^2 \right)^{1-C\delta^2} \right),
\end{align} where $r={\rm unc}(\mathbb R^m, \|\cdot\|)$ and $c,C>0$ are absolute constants.
\end{theorem}

This result almost reaches the final goal, except from the parameter $r ={\rm unc} X$. 
Note that John's theorem \cite{Jo} readily implies that ${\rm unc} X \ls \sqrt{\dim X}$ for any 
finite dimensional normed space. On the other hand, 
it is known (see e.g. \cite{FKP} and \cite{FJ}) that a ``typical" subspace $F$ of $\ell_\infty^n$ of proportional dimension satisfies 
${\rm unc}F \gtrsim \sqrt{\dim F}$, which marks the end of usefulness of Theorem \ref{thm:opt-r-unc} in those cases. 
Nonetheless, we may overcome this obstacle by locating moderate unconditional structure in {\it every} normed space.
In order to do so, our argument breaks into a dichotomy based on the Alon--Milman theorem \cite{AM}, 
and its sharper form due to Talagrand \cite{Tal-am}. Below we state it as a lemma 
in a customized form that will be of immediate use for us. 

\begin{lemma} \label{lem:amtdr}
Let $X=(\mathbb R^n, \|\cdot\|)$ be a normed space and let $B_2^n=\{x: \|x\|_2\ls 1\}$ be the maximal volume ellipsoid inscribed in 
$B_X=\{x: \|x\|\ls 1\}$. Then, there exists $\sigma\subset [n]$ with $|\sigma| \gr c\sqrt{n}$ such that 
\begin{align*}
\frac{1}{8} \max_{i\in \sigma}|\alpha_i| \ls \left\| \sum_{i\in \sigma} \alpha_i e_i  \right\| \ls 4\sqrt{k} \max_{i\in \sigma}|\alpha_i|,
\end{align*} for all scalars $(\alpha_i)_{i\in \sigma} \subset \mathbb R$, where $k=k(\|\cdot\|)$ is the Dvoretzky number 
defined in \eqref{eq:conc-dvo}.
\end{lemma}

The reader is referred to \cite[Theorem 4.4]{PV-dicho} for the detailed proof of this formulation which uses the classical Dvoretzky-Rogers
lemma \cite{DR}.

Combining Lemma~\ref{lem:amtdr} with Theorem \ref{thm:opt-r-unc}, we can prove Theorem \ref{thm:main-2a}.

\medskip

\noindent{\it Proof of Theorem~\ref{thm:main-2a}.} We may assume that $B_2^n$ is the 
maximal volume ellipsoid inscribed in $B_X$ and let
$k = (\mathbb E\|G\|)^2 \ls n^{1/4}$, otherwise if $k >n^{1/4}$ the assertion follows from \eqref{eq:conc-dvo} 
(with $T={\rm Id}$). Then, Lemma~\ref{lem:amtdr} 
yields a $\sigma \subset [n]$ with $|\sigma|\asymp \sqrt{n}$ such that
\begin{align}
c_1 \max_{i\in \sigma} |\alpha_i| \ls \left\| \sum_{i\in \sigma} \alpha_i e_i \right\| \ls C_1 \sqrt{k} \max_{i\in \sigma} |\alpha_i|,
\end{align} for all scalars $(\alpha_i)\subset \mathbb R$. Set $E_\sigma={\rm span}\{e_i: i\in \sigma\}$ and note that the above
estimate shows that 
\begin{align*}
r:={\rm unc}(E_\sigma, \|\cdot\|)\ls C_2 \sqrt{k} \ls C_2n^{1/8}.
\end{align*}
Therefore, Theorem \ref{thm:opt-r-unc} yields the existence of $S\in GL(E_\sigma)$ with the property that for all $\delta\in (0,1/2)$ one has
\begin{align*}
\mathbb P \left\{ \|SZ\| \ls \delta \mathbb E\|SZ\| \right\} \ls \exp\left( - ( |\sigma|/ r^2)^{1-C \delta^2 } \right) 
\ls \exp \left( - c  n^{\frac{1}{4}-C' \delta^2 } \right), \quad Z \sim N({\bf 0},I_{E_\sigma}).
\end{align*} Now we apply a ``lifting" argument: We augment the linear map $S$ by considering 
$T=S\oplus (a I_{E_{\sigma}^\perp})$ to get
\begin{align*} 
\mathbb E\|TG\| \ls 2\mathbb E\|SZ\|,
\end{align*} provided $a>0$ is sufficiently small and, moreover,
\begin{align*}
\mathbb P\left( \|TG\| \ls t \right) \ls \mathbb P \left( \|SZ\| \ls t \right), \quad t>0,
\end{align*} The result readily follows. \prend

\subsection{Dvoretzky-type results}

Although the connection of local almost Euclidean structure with the Gaussian concentration has its origins in V. Milman's work
\cite{Mil-dvo}, the link to the Gaussian small ball estimate had not been considered until the work 
of Klartag and Vershynin in \cite{KV}. 
The authors there establish the remarkable phenomenon that any centrally symmetric convex body in $\mathbb R^n$, admits
random sections of dimension larger than the Dvoretzky number which are well bounded in terms of the expected radius. 

\begin{theorem} [Klartag--Vershynin]
Let $A$ be a centrally symmetric convex body in $\mathbb R^n$. Then, for any $1\ls m\ls cd(A)$, the random 
$m$-dimensional subspace $F$ (with respect to the Haar probability measure on the Grassmannian $G_{n,m}$) satisfies
\begin{align*}
A\cap F \subseteq \frac{C}{M(A)} B_F, \quad M(A)= \int_{S^{n-1}} \|\theta\|_A \, d\sigma(\theta) = (\mathbb E\|G\|_2)^{-1} \mathbb E\|G\|_A,
\end{align*} with probability greater than $1-e^{-cd(A)}$.
Equivalently, the random $F$ satisfies
\[
\|y\|_A \gr c\frac{\mathbb E\|G\|_A}{\sqrt{n}} \|y\|_2, \quad \forall \, y\in F,
\] with overwhelming probability, where $\|\cdot\|_A$ is the induced norm.
\end{theorem}

In order to make the connection of the small ball estimate with the one-sided random version of Dvoretzky's theorem more transparent, 
let us first recall the definition of the lower Dvoretzky dimension $d$ due to Klartag and Vershynin:
For a centrally symmetric convex body $A$ in $\mathbb R^n$ and $\delta\in (0,1)$, one defines
\footnote{
Let us point out that the original formulation uses the uniform probability measure $\sigma$ on the sphere $S^{n-1}$. However, one 
can equivalently work with the Gaussian measure for this problem in the light of the following (standard) estimates that compare the 
two measures:
\begin{align*}
\gamma_n(tB_2^n) \sigma(A \cap S^{n-1}) \ls \gamma_n(t A), \quad \gamma_n(sA) \ls \gamma_n(sB_2^n) +\sigma(A\cap S^{n-1}), \quad t,s>0,
\end{align*} where $A$ is any centrally symmetric convex body in $\mathbb R^n$; see in particular \cite{KV} and \cite{LO}.
}
\begin{align*}
d(A,\delta) = \min \left\{ n, -\log \mathbb P \left\{ \|G\|_A \ls \delta {\rm med}(\| G \|_A)\right\} \right \},
\end{align*} where ${\rm med}(\|G\|_A)$ is the median of $\|G\|_A$.
For $\delta=1/2$ we simply write $d(A)=d(A,1/2)$. Building on the ideas
of Latala and Oleszkiewicz from \cite{LO} and employing the $B$-inequality \cite{CFM} they prove the following:
\begin{align} \label{eq:KV-sb}
\mathbb P \left\{ \|G\|_A \ls c\varepsilon \mathbb E\|G\|_A \right\} \ls \mathbb P \left\{ \|G\|_A \ls \varepsilon {\rm med}(\|G\|_A)\right\} 
\ls \varepsilon^{\frac{d(A,\delta)-\log 2}{\log(1/\delta)}}, \quad \varepsilon\in (0,\delta), \quad (0<\delta <1/2),
\end{align} where we have also used the standard fact that $\mathbb E\|G\| \ls C {\rm med}(\|G\|)$ for any norm $\|\cdot\|$.

The next step is to express small ball estimates in terms of reverse H\"older inequalities. More precisely, the estimate
\eqref{eq:KV-sb} implies the following standard: 

\begin{lemma} \label{lem:neg-moms}
	Let $A$ be a centrally symmetric convex body in $\mathbb R^n$ and let $\delta\in (0,1/3)$. Then,
\begin{align*}
\left( \mathbb E\|G\|_A^{-q} \right)^{-1/q} \gr c \delta \mathbb E\|G\|_A, \quad q\ls c d(A,\delta)/\log(1/\delta).
\end{align*} 
\end{lemma}

\noindent {\it Sketch of Proof.} We write
\[
\mathbb E\|G\|_A^{-q} = q ({\rm med}(\|G\|_A))^{-q}\int_0^\infty t^{-q-1} \mathbb P\{\|G\|_A \ls t {\rm med}(\|G\|_A) \} \, dt
\]
and we use \eqref{eq:KV-sb} to bound the integral as follows
\begin{align*}
q\int_0^\infty t^{-q-1} \mathbb P\{\|G\|_A \ls t {\rm med}(\|G\|_A) \} \, dt &\ls q \int_0^\delta t^{a-q-1} \, dt + q \int_\delta^\infty t^{-q-1}\, dt \\
&\ls \frac{q}{a-q}\delta^{a-q} + \delta^{-q} \ls \delta^{-q} \frac{a}{a-q} \ls 2\delta^{-q},
\end{align*} as long as $2q\ls a: = \frac{d(A,\delta)-\log 2}{\log (1/ \delta)}$. \prend

\medskip

The final step requires to bound the diameter of random sections in terms of  the negative moments of $\|G\|_A$. Note that this 
cannot be achieved now by the standard net argument. The reason is that the latter works once we have first established an upper bound
for the norm in the subspace, equivalently that the section contains relatively large ball; this is not the case in this regime. To overcome
this obstacle Klartag and Vershynin devise a dimension lift technique based on the low $M$-inequality from \cite{K-geom}.

\begin{lemma} [Dimension lift] \label{lem:dim-lift}
Let $A$ be a centrally symmetric convex body in $\mathbb R^n$. Then, for any $1\ls m\ls q \ls n/8$ we have
\begin{align}
\left( \int_{G_{n,m}} {\rm diam}(A\cap F)^m \, d\nu_{n,m} (F) \right)^{1/m} \ls 
\frac{C}{M(A)}\left( \frac{\mathbb E\|G\|_A}{ \left( \mathbb E\|G\|_A^{-q} \right)^{-1/q} } \right)^2.
\end{align}
\end{lemma}

The latter shows that the diameter of a random section is controlled by the tightness of the reverse H\"older inequality for the negative moments.
In view of Lemma \ref{lem:neg-moms}, Lemma \ref{lem:dim-lift}
and the findings of this paper we are able to prove the following (geometric reformulation of Theorem \ref{thm:main-4}):

\begin{theorem} \label{thm:main-4a}
Let $A$ be a centrally symmetric convex body in $\mathbb R^n$ and let $ \| \cdot \|_A$  be the induced norm. 
\begin{enumerate}
	\item If $\|\cdot\|_A$ is $1$-unconditional in the position of minimal $M$, or $\ell$-position, or $w^{1,1}$--position, then 
	for any $\delta\in (0,1/3)$ and $k\ls cn^{1-C\delta^2}$, the random $k$-dimensional subspace $E$ of $\mathbb R^n$ satisfies
	\[
	A \cap E \subseteq \frac{C\delta^{-2}}{M(A)} B_E,
	\] with probability greater than $1-e^{-cn^{1-C\delta^2}}$.
	
	\item In the general case, there exists an invertible linear map $T$ with the following property: for any $\delta\in (0, 1/3)$
	and $m \ls n^{1/4-C\delta^2}$, 
	the random $m$-dimensional subspace $F$ satisfies
	\begin{align*} \label{eq:main-3}
	TA \cap F \subseteq  \frac{C\delta^{-2}}{ M(TA)} B_F ,
	\end{align*} with probability greater than $1-e^{-cn^{1/4-C\delta^2}}$, where $C,c>0$ are universal constants. 
\end{enumerate}
\end{theorem}

\paragraph{\bf Quermassintegrals.} Similar Dvoretzky-type results can be proved for 
other geometric quantities associated with the convex body than the negative moments of 
the norm. For example, as an application of the  dimension lift (Lemma \ref{lem:dim-lift}) and the Gaussian deviation inequality (Theorem~\ref{thm:sdi}) 
it is proved in \cite{PPV} a quantitive  reversal of the classical Alexandrov inequality. The latter says that the sequence of  quermassintegrals of a convex body
$A$, appropriately normalized, is monotone. Recall that the $k$-th quermassintegral of $A$ is the average of volume of the random $k$-dimensional projection
of $A$, that is
\[
Q_k(A) := \left( \frac{1}{{\rm vol}(B_2^k)} \int_{G_{n,k}} {\rm vol}_k(P_F A) \, d\nu_{n,k}(F)\right)^{1/k},
\] where $\nu_{n,k}$ is the (unique) Haar probability measure invariant under the orthogonal group action on the Grassmannian $G_{n,k}$. With
this normalization Alexandrov's inequality reads as follows: 
\[
Q_n(A) \ls Q_{n-1}(A) \ls \ldots \ls Q_1(A).
\] Note that $Q_n(A)$ is the volume radius of $A$, $Q_{n-1}(A)$ is (an appropriate multiple of) the surface area of $A$, while $Q_1(A)$ stands for 
the mean width. We refer the reader to Schneider's monograph \cite{Schn} for related background material and an excellent exposition 
in convex geometry. The main result of \cite{PPV} asserts that a long initial segment of the above sequence is essentially constant, where the length and the almost constant behavior is quantified in terms of the parameter $1/\beta(A^\circ)$, namely
\[
Q_k(A) \gr \left( 1-\sqrt{k\beta(A^\circ) \log \left(\frac{c}{k\beta(A^\circ)}\right) } \right) Q_1(A),
\] where $A^\circ$ is the polar body of $A$.

The strong small-ball probabilities established in the present paper yield an isomorphic reversal (up to constant) which extends 
to a polynomial order length, regardless the order of magnitude of $\beta(A)$. More precisely, we have the following:

\begin{theorem}
Let $A$ be a centrally symmetric convex body in $\mathbb R^n$. Let $A^\circ$ be its polar body, that is 
\[
A^\circ=\{x\in \mathbb R^n : \langle x,y \rangle \ls 1 \, \forall y\in A\}.
\]
Then, 
\begin{enumerate}
	\item If $A$ is $1$--unconditional, and $A^\circ$ is in the $\ell$-position, or position of minimal $M$, or $w^{1,1}$--position, 
	then for any $\delta\in (0,1/2)$ and for any $k\ls cn^{1-C\delta^2}$ we have
	\[
	Q_k(A) \gr c \delta^2 Q_1(A).
	\]
	\item In the general case there exists a linear image $\widetilde A$ of $A$ such that for any $\delta\in (0,1/2)$ and for any $k\ls cn^{1/4-C\delta^2}$ 
	we have
	\[
	Q_k(\widetilde A) \gr c \delta^2 Q_1(\widetilde A).
	\]
\end{enumerate}
\end{theorem}

\noindent {\it Proof.} We prove only (1). The other case follows in similar fashion. 
We apply Theorem \ref{thm:main-4} in the dual setting to the polar body $A^\circ$ since $A^\circ \cap F = P_F(A)$ for 
any subspace $F$. To this end, let $A^\circ$ be in one of
the announced positions and let $\delta\in (0,1/2)$. Then, for $k\ls cn^{1-C\delta^2}$ and by taking into account the observation
$Q_1(A) \asymp M(A^\circ)$ we derive that the set
\[
{\mathcal F} = \left\{ F\in G_{n,k} : P_F(A) \supseteq c\delta^2  Q_1(A) B_F \right\},
\] has measure $\nu_{n,k}(\mathcal F) \gr 1-e^{-cn^{1-C\delta^2}}$ by Theorem  \ref{thm:main-4}. Therefore,
\begin{align*}
Q_k(A) &= \left( \frac{1}{{\rm vol}(B_2^k)} \int_{G_{n,k}} {\rm vol}_k(P_F A) \, d\nu_{n,k}(F)\right)^{1/k} \\
&\gr \left( \frac{1}{{\rm vol}(B_2^k)} \int_{\mathcal F} {\rm vol}_k(P_F A) \, d\nu_{n,k}(F)\right)^{1/k} \\
&\gr c\delta^2 Q_1(A) \nu_{n,k}(\mathcal F)^{1/k} \gr c' \delta^2 Q_1(A),
\end{align*} which proves the assertion. \prend

\newpage

\section{Small deviations for norms}

The goal of this section is to obtain lower deviation estimates
$\Prob\big\{ \|G\|< (1-\varepsilon) \Exp\|G\| \big\}$, $\varepsilon\in(0,1/2]$,
for norms $\|\cdot\|$ satisfying certain ``balancing'' conditions.
The following two theorems are consequences of the main technical result of the section (Theorem~\ref{th: smalldev main}).
\begin{theorem}\label{th: smdev unc}
For any $1$--unconditional norm $\|\cdot\|$ in $\R^n$
in the position of minimal $M$, or $\ell$--position, or $w^{1,1}$--position,
we have
$$\Prob\big\{ \|G\|< (1-\varepsilon) \Exp\|G\| \big\}\ls 3\exp(-n^{c\varepsilon}),\quad \varepsilon\in(0,1/2],$$
where $c>0$ is a universal constant.
\end{theorem}
For the standard $\ell_\infty^n$--norm in $\R^n$, a
reverse estimate is known \cite{Sch-cube}:
$$c\exp( - n^{C\varepsilon}) < \Prob\big\{ \|G\|_\infty< (1-\varepsilon) \Exp\|G\|_\infty \big\},\quad \varepsilon\in(0,1/2],$$
which implies that the above result is optimal (in an appropriate sense).

For general norms, we are able to obtain corresponding deviation estimates in a certain (non-classical) position:
\begin{theorem}\label{th: smdev gen}
For any norm $\|\cdot\|$ in $\R^n$ there is an invertible linear transformation $T$ such that
$$\Prob\big\{ \|T(G)\|< (1-\varepsilon) \Exp\|T(G)\| \big\}\ls 3\exp(-n^{c\varepsilon}),\quad \varepsilon\in(0,1/2],$$
where $c>0$ is a universal constant.
\end{theorem}

\bigskip

The proof of the main result follows a different strategy compared to the results of Section~\ref{s: Small-ball probability estimates}.
To highlight the basic technical issues with attempting to reuse the relation
$P_tf(x) \ls e^{-t} f(x) +\sqrt{1-e^{-2t}}\mathbb Ef(G)$ together with Theorem~\ref{thm:sdh} in the context of small deviations,
consider the standard $\ell_\infty$--norm in $\R^n$. Assume we want
to show
that $\Prob\big\{ \|G\|_\infty< (1-\varepsilon) \Exp\|G\|_\infty \big\}\ls 3\exp(-n^{c\varepsilon})$ for all $\varepsilon\in(0,1/2]$.
Using the relation and the  theorem, we may write
\begin{align*}
\mathbb P \left( \|G\|_\infty \ls (1-\varepsilon) \mathbb E \|G\|_\infty \right) 
&\ls \mathbb P \left( P_t\|G\|_\infty \ls \left( (1-\varepsilon) e^{-t}+\sqrt{1-e^{-2t}} \right) \mathbb E [P_t\|G\|_\infty] \right)\\
&=\mathbb P \left( P_t\|G\|_\infty \ls \mathbb E [P_t\|G\|_\infty]-
y\,\mathbb E [P_t\|G\|_\infty]\right)\\
&\ls \exp\left(-c\, y^2\log(n)/\big(\mathbb E\|\nabla\, P_t\|G\|_\infty\|_2^2 \big)\right),
\end{align*}
where $t>0$ is any number such that $y:=1-(1-\varepsilon) e^{-t} -\sqrt{1-e^{-2t}}\gr 0$.
It is not difficult to see that when $\varepsilon$ is small,
the last inequality forces one to take $t$ of order at most $O(\varepsilon^2)$,
with $y$ being at most of order $\varepsilon$. The estimate then becomes too weak to yield an optimal result.

Although the operator $P_t$ still plays a fundamental role in this section, instead of ``replacing'' the original norm with
$P_t\|G\|$ (as is done in the proofs of Theorems~\ref{thm:hyper-sb} and~\ref{thm:super-sb}),
we will replace it with an auxiliary seminorm by carefully choosing a subset of norming functionals for $\|\cdot\|$,
in such a way that the seminorm satisfies strong lower deviation estimates.
In that respect, the proof is similar to the approach from Subsection~\ref{subs: lm position},
although the actual construction is completely different.

The strategy of proving Theorems~\ref{th: smdev unc} and~\ref{th: smdev gen} can roughly be described
as follows.
Given a norm $f$ in a ``balanced'' position, we consider two situations. 
\begin{itemize}
\item If the variance of the norm $f$, or the expected
squared length of the gradient, is
small compared to $(\Exp f(G))^2$ then 
the statements from Section~\ref{s: Small-ball probability estimates} give the required estimate (in fact, a much stronger inequality).
\item Otherwise, one is able to construct a seminorm $\semi(\cdot)$ dominated by $f$, with a very small Lipschitz constant, and
with $\Exp \semi(G)$ very close to $\Exp f(G)$. The lower deviation probabilities for $f$ can then be 
bounded by corresponding probabilities for $\semi(G)$, which in turn can be efficiently estimated since
the Lipschitz constant of the seminorm $\semi(\cdot)$ is small.
\end{itemize}
The reason why the dichotomy works (i.e every appropriately ``balanced'' norm must satisfy at least one of one of the two above conditions)
is rooted in the hypercontractivity properties of the Gaussian measure.
This will be made rigorous further.

\subsection{Construction of seminorms}\label{subs: seminorms}
Let $H\gr 1$, $\tau\in[\log^{-1} n,1/2]$ and $\delta\in(0,1]$ be parameters.
Given a norm $f(\cdot)=\|\cdot\|$ in $\R^n$ with $H\gr\unc(\|\cdot\|,\{e_i\}_{i=1}^n)$,
and
\begin{equation}\label{eq: ndelta cond}
n^{-\delta}\gr \frac{1}{R(f)} = \frac{\sum_{i=1}^n (\Exp | \partial_i f(G) |)^2 }{\Exp \|\nabla f(G)\|_2^2},
\end{equation}
and a collection of Borel subsets $(\Event_t)_{t\gr \tau}$ of $\R^n$,
define the seminorm $\semi(\cdot)$ as follows:
For every $t\gr\tau$, set
\begin{equation}\label{eq: F definition}
F_t:=\big\{x\in\R^n:\;\langle x,P_t\nabla f(x) \rangle\gr (1-4t)\Exp f(G)\quad\mbox{and}\quad
\|P_t \nabla f(x)\|_2\ls \|\nabla f\|_{L^2(\gamma_n)}\,n^{-\delta t/8}\big\},
\end{equation}
and
$$w(x):=\sup\big\{\langle x,P_t \nabla f(y)\rangle:\;y\in F_t\setminus \Event_t,\;t\gr \tau\big\},\quad x\in\R^n$$
(if $F_t\setminus \Event_t=\emptyset$ for all $t\gr \tau$ then we set $w\equiv 0$).
Then for every vector $x=(x_1,x_2,\dots,x_n)\in\R^n$ we define
$$\semi(x):=\max\Big(w(x),\frac{1}{H}\max\limits_{\epsilon_1,\dots,\epsilon_n\in\{-1,1\}}w\Big(\sum_{i=1}^n\epsilon_i x_ie_i\Big)\Big).$$
The next statement follows from the definition:
\begin{lemma}\label{l: apiurls;,m} For any parameters $H,\tau,\delta$ and Borel subsets $(\Event_t)_{t\gr \tau}$,
the seminorm $\semi(\cdot)$ satisfies
\begin{itemize}
\item $\unc(\semi(\cdot),\{e_i\}_{i=1}^n)\ls H$;
\item $\semi(x)\ls f(x)$ for all $x\in\R^n$;
\item $\Lip(\semi)\ls \Lip(f)\,n^{-\delta \tau/8}$;
\item $\semi(x)\gr (1-4t)\Exp f(G)$ for every $t\gr \tau$ and $x\in F_t\setminus \Event_t$.
\end{itemize}
\end{lemma}
While the first, third and fourth assertions are straightforward, we remark that the second assertion $\semi(x)\ls f(x)$, $x\in\R^n$, follows since for every $x\in\R^n$ the quantity
$\langle x,P_t\nabla f(x) \rangle$ is a convex combination of inner products $\langle x,u\rangle$, with $u$ taking values in the unit ball of the dual space for $(\R^n,f)$, and therefore
$\langle x,P_t\nabla f(x) \rangle\ls f(x)$.

Note that if $F_t$ and $\Event_t$ are constructed in such a way that $\bigcup\limits_{t\gr \tau}(F_t\setminus \Event_t)$ has the standard Gaussian measure close to one, the fourth assertion of the lemma
implies a strong lower bound for the expectation $\Exp\,\semi(G)$. This, together with the condition $\semi(\cdot)\ls f(\cdot)$ would allow to reduce the problem of estimating
$\Prob\{f(G)< (1-\varepsilon) \Exp\,f(G)\}$ to bounding a quantity $\Prob\{\semi(G)< (1-\tilde c\varepsilon) \Exp\,\semi(G)\}$, which can be done by using the condition that the Lipschitz constant of $\semi(\cdot)$
is very small. The crucial assumption \eqref{eq: ndelta cond} is introduced specifically to guarantee that the sets $F_t$ and $\Event_t$ satisfying the needed conditions can be found (as we will show later).

We will view the above procedure as a construction of a family $\mathcal F$ of seminorms parameterized by
$H,\tau,\delta,f$ and Borel sets $(\Event_t)_{t\gr\tau}$.
In what follows, by $\semi(\cdot)$ we will understand a particular seminorm from the family $\mathcal F$ for a given realization of
parameters (which are clear from the context and are either given explicitly or required to take values
from a specific range).

\subsection{Lower bound for the expectation $\Exp\semi(G)$}
In what follows, $G$ and $Z$ denote independent standard Gaussian vectors in $\R^n$,
and for any $t>0$, by $G_t$ we denote the random vector $e^{-t}G+\sqrt{1-e^{-2t}}Z$.
The goal of this subsection is to prove

\begin{proposition}\label{p: piul,ams}
Let the norm $f(\cdot)$ satisfy \eqref{eq: ndelta cond}, as well as the conditions $\delta\sqrt{\log n}\,\Lip(f)\ls \Exp f(G)$ and
$$\Prob\big\{|f(G)-\Exp f(G)|\gr s\, \Exp f(G)\big\}\ls 2\exp(-\delta s\log n),\quad s\gr 0,$$
for some $\delta\in(0,1]$.
Let $H>0$, $\tau\in[\log^{-1}n,1/2]$ and
assume that the Borel sets $\Event_t$, $t\in[\tau,1/2]$, satisfy $\Prob(\Event_t)\ls 2n^{-\delta t/4}$.
Then the seminorm $\semi(\cdot)$ defined above with parameters
$H,\tau,\delta$, $(\Event_t)_{t\gr \tau}$, satisfies
$$\Exp \semi(G)\gr (1-C_{\ref{p: piul,ams}}\delta^{-2}\tau)\Exp f(G)$$
for some universal constant $C_{\ref{p: piul,ams}}>0$.
\end{proposition}
\begin{remark}
We note that if the assumptions are satisfied, the above proposition, combined with \eqref{eq:conc-L}, allows to obtain the optimal
lower deviation inequality in the regime $\varepsilon\in[C'\log\log n/\log n,1/2]$
(taking $\Event_t=\emptyset$ for all admissible $t$), namely
$$\Prob\big\{f(G)\ls (1-\varepsilon)\Exp f(G)\big\}\ls 3\exp(-n^{c'\varepsilon}),$$
for some $c',C'>0$ depending only on $\delta$.
Indeed, setting $\tau:=\delta^2\varepsilon/(2C_{\ref{p: piul,ams}})$, we get
\begin{align*}
\Prob\big\{f(G)\ls (1-\varepsilon)\Exp f(G)\big\}
&\ls \Prob\big\{\semi(G)\ls (1-\varepsilon)(1-C_{\ref{p: piul,ams}}\delta^{-2}\tau)^{-1}\Exp \semi(G)\big\}\\
&\ls \exp\bigg(-\frac{(\Exp \semi(G))^2\big(1-(1-\varepsilon)(1-C_{\ref{p: piul,ams}}\delta^{-2}\tau)^{-1}\big)^2}{2\,\Lip(\semi)^2}\bigg),
\end{align*}
where, in view of Lemma~\ref{l: apiurls;,m} and the assumptions of Proposition~\ref{p: piul,ams},
$$\Lip(\semi)^2\ls \Lip(f)^2\,n^{-\delta \tau/4}\ls \delta^{-2}(\log n)^{-1}(\Exp f(G))^2
\,n^{-\delta \tau/4},$$
and where for all $\varepsilon\ls 1/2$, we have $1-(1-\varepsilon)(1-C_{\ref{p: piul,ams}}\delta^{-2}\tau)^{-1}\gr c''\varepsilon$ for a universal constant $c''>0$.
Thus, we get
$$
\Prob\big\{f(G)\ls (1-\varepsilon)\Exp f(G)\big\}
\ls \exp\bigg(-\frac{(c''\delta)^2 (\log n)\varepsilon^2\,n^{\delta^3\varepsilon /(8C_{\ref{p: piul,ams}})}}{8}\bigg).
$$
When $\varepsilon\gr C'\log\log n/\log n$ for a sufficiently large $C'$ depending on $\delta$, we have $\frac{(c''\delta)^2 (\log n)\varepsilon^2\,n^{\delta^3\varepsilon /(16C_{\ref{p: piul,ams}})}}{8}\gr 1$,
so the last expression is bounded from above by $\exp\big(-n^{\delta^3\varepsilon /(16C_{\ref{p: piul,ams}})}\big)$ implying the result.
Note, on the other hand, that in the regime $\varepsilon\ll \log\log n/\log n$, the quantity $(\log n)\varepsilon^2$ is much smaller than $n^{-\delta^3\varepsilon /(8C_{\ref{p: piul,ams}})}$, so the above computations
do not yield a non-trivial bound.

\medskip

In order to extend the range to all $\varepsilon\in(0,1/2]$,
we would need an estimate which would cancel the influence of $\varepsilon^2$ in the power of the exponent.
If it were true that $\semi(\cdot)$ is superconcentrated in the sense that $\Var\semi(G)=O(\log^{-1} n)\, \Lip(\semi)^2$, we would be able to complete the proof given the assumptions of Proposition~\ref{p: piul,ams}.
Indeed, in that case, by Theorem~\ref{thm:sdi},
$$\Prob\big\{f(G)\ls (1-\varepsilon)\Exp f(G)\big\}
\ls \Prob\big\{\semi(G)\ls (1-\varepsilon)(1-C\delta^{-2}\tau)^{-1}\Exp \semi(G)\big\}
\ls 2 \exp\big(-c\varepsilon^2 (\Exp \semi(G))^2/\Var\semi(G)\big),$$
where $\Var\semi(G)= O(\log^{-1} n)\,  \Lip(\semi)^2=O(\delta^{-2}(\log n)^{-2})\,(\Exp f(G))^2
\,n^{-\delta \tau/4}.$ Essentially repeating the above computations, we would then get
$$
\Prob\big\{f(G)\ls (1-\varepsilon)\Exp f(G)\big\}\ls 2\exp\big(-\hat c (\log n)^2\varepsilon^2\,n^{\hat c\varepsilon}\big)
$$
for some $\hat c>0$ depending only on $\delta$. For $\varepsilon> \frac{1}{\sqrt{\hat c}\log n}$, the quantity $\hat c (\log n)^2\varepsilon^2$ is greater than one,
so the last bound transforms to the required estimate $2\exp(-n^{\hat c\varepsilon})$. For $\varepsilon\ls \frac{1}{\sqrt{\hat c}\log n}$,
on the other hand, the probability estimate $3\exp(-n^{c'\varepsilon})$ is trivial as long as $c'>0$ is chosen sufficiently small.

In order to obtain the strong bounds on the variance
$\Var\semi(G)$, we will need $\semi(\cdot)$ to be ``well balanced'' in an appropriate sense.
Note that although the norm $f$ is assumed to be ``well balanced'' there is no trivial argument which would guarantee analogous conditions for the seminorm $\semi(\cdot)$.
Thus, our main goal is to carefully choose
the events $\Event_t$ in the definition of $\semi(\cdot)$ that would provide the required properties.
That will be done in the next subsections.
\end{remark}

\begin{lemma}\label{l: aux a;oia;slkfmsflm}
Assume that a norm $f(\cdot)$ in $\R^n$ satisfies $\delta\sqrt{\log n}\,\Lip(f)\ls \Exp f(G)$ for some parameter
$\delta>0$. Then for any $t>0$ and $x\in\R^n\setminus\{0\}$ such that the gradient of $f$ at $x$ is well defined, we have
$$\Prob\big\{\langle e^{-t}x+\sqrt{1-e^{-2t}}Z,\nabla f(x)
\rangle\ls e^{-t}f(x)-\delta^{-1}\sqrt{t/\log n}\,u\,\Exp f(G)\big\}\ls 2\exp(-cu^2),\quad u>0,$$
where $c>0$ is a universal constant.
\end{lemma}
\begin{proof}
Obviously,
$$\langle e^{-t}x+\sqrt{1-e^{-2t}}Z,\nabla f(x)\rangle=e^{-t}f(x)
+\sqrt{1-e^{-2t}}\langle Z,\nabla f(x)\rangle.$$
The inner product
$\langle Z,\nabla f(x)\rangle$ is equidistributed with $g\|\nabla f(x)\|_2$, where $g$ is a standard Gaussian variable.
Further, $\|\nabla f(x)\|_2\ls\Lip(f)\ls \Exp f(G)/(\delta\sqrt{\log n})$, by the assumptions
of the lemma. The result follows.
\end{proof}

\begin{lemma}\label{l: apon,mfpauh}
Let the norm $f(\cdot)$ satisfy $\delta\sqrt{\log n}\,\Lip(f)\ls \Exp f(G)$ and
$$\Prob\big\{|f(G)-\Exp f(G)|\gr s \Exp f(G)\big\}\ls 2\exp(-\delta s\log n),\quad s\gr 0,$$
for some parameter $\delta\in(0,1]$. Then for any $u\gr 1$ and $t\in[\log^{-1} n,1/2]$ we have
$$\Prob\big\{\langle G,P_t \nabla f(G)\rangle \ls e^{-t}f(G)-\delta^{-1}\sqrt{t/\log n}\,u\,\Exp f(G)\big\}
\ls 2\exp(-c u),$$
for some universal constant $c>0$.
\end{lemma}
\begin{proof}
Fix any $u\gr 1$.
Define a subset $S$ of $\R^n$ by setting
$$S:=\big\{x\in\R^n:\;\langle x,P_t \nabla f(x)\rangle \ls e^{-t}f(x)-\delta^{-1}\sqrt{t/\log n}\,u\,\Exp f(G)\big\}.$$
Fix for a moment any $x\in S$.
By definition,
$$e^{-t}f(x)-\delta^{-1}\sqrt{t/\log n}\,u\,\Exp f(G)\gr \langle x,P_t \nabla f(x)\rangle=
\Exp_Z\langle x,\nabla f({e^{-t}x+\sqrt{1-e^{-2t}}Z})\rangle,
$$
where, as usual, $Z$ is the standard Gaussian vector independent from $G$.
The upper bound on the expectation then implies that there is an integer $i=i(x)\gr 1$ with
\begin{equation}\label{eq: 09482059870}
\Prob_Z\big\{\langle x,\nabla f({e^{-t}x+\sqrt{1-e^{-2t}}Z})\rangle\ls
e^{-t}f(x)-ci\delta^{-1}\sqrt{t/\log n}\,u\,\Exp f(G)\big\}\gr \frac{1}{i^2},
\end{equation}
for a sufficiently small universal constant $c>0$
(instead of $\frac{1}{i^2}$ on the right hand side, we can take the $i$-th element of arbitrary convergent series).
Indeed, this immediately follows from the obvious relation
\begin{align*}
\Exp \xi&\gr (a-b)\Prob\{\xi>a-b\}+\sum\limits_{i=1}^\infty \big(a-(i+1)b\big)\Prob\{a-(i+1)b<\xi\ls a-ib\}\\
&\gr a-b-2b\sum\limits_{i=1}^\infty\Prob\{\xi\ls a-ib\},
\end{align*}
which is valid for any variable $\xi$ and any parameters $a\in\R$ and $b>0$.

For a given $i\gr 1$, denote the collection of all $x\in S$ satisfying \eqref{eq: 09482059870} by $S_i$.
Since $\bigcup\limits_{i=1}^\infty S_i=S$, we get that there is $i_0\gr 1$
with $\gamma_n(S_{i_0})\gr c'\gamma_n(S)/i_0^2$.
Thus,
$$\Prob\big\{\langle G,\nabla f({G_t})\rangle\ls
e^{-t}f(G)-ci_0\delta^{-1}\sqrt{t/\log n}\,u\,\Exp f(G)\big\}\gr \frac{c'}{i_0^4}\gamma_n(S),$$
where $G_t:=e^{-t}G+\sqrt{1-e^{-2t}}Z$.
This implies
\begin{align*}
\Prob\big\{&\langle G,\nabla f({G_t})\rangle\ls
e^{-t}f(G_t)-\mbox{$\frac{c}{2}$}i_0\delta^{-1}\sqrt{t/\log n}\,u\,\Exp f(G)\big\}\\
&\gr \frac{c'}{i_0^4}\gamma_n(S)
-\Prob\big\{|f(G)-f(G_t)|\gr \mbox{$\frac{c}{2}$}i_0\delta^{-1}\sqrt{t/\log n}\,u\,\Exp f(G)\big\}.
\end{align*}
Observe that the probability on the left hand side of the inequality is equal to
$$\Prob\big\{\langle G_t,\nabla f(G)\rangle\ls
e^{-t}f(G)-\mbox{$\frac{c}{2}$}i_0\delta^{-1}\sqrt{t/\log n}\,u\,\Exp f(G)\big\},$$
and in this form can be estimated with help of Lemma~\ref{l: aux a;oia;slkfmsflm}.
This gives
$$\frac{c'}{i_0^4}\gamma_n(S)\ls
\Prob\big\{|f(G)-f(G_t)|\gr \mbox{$\frac{c}{2}$}i_0\delta^{-1}\sqrt{t/\log n}\,u\,\Exp f(G)\big\}
+2\exp(-\tilde ci_0^2u^2).$$
It remains to note that the condition on the concentration of $f(G)$ implies
$$\Prob\big\{|f(G)-f(G_t)|\gr \mbox{$\frac{c}{2}$}i_0\delta^{-1}\sqrt{t/\log n}\,u\,\Exp f(G)\big\}
\ls 2 e^{-c'' i_0u\sqrt{t\log n}},$$
and solve the inequality for $\gamma_n(S)$.
\end{proof}

\begin{lemma}\label{l: ooyeksn}
Let $f(\cdot)$ be a norm in $\R^n$ satisfying \eqref{eq: ndelta cond}. Then for any $t\in(0,1/2]$ we have
$$\Prob\big\{\|P_t \nabla f(G)\|_2\ls \|\nabla f\|_{L^2(\gamma_n)}\,n^{-\delta t/8}\big\}\gr 1-n^{-c\delta t},$$
where $c>0$ is a universal constant. 
\end{lemma}
\begin{proof}
Repeating the argument from the proof of Proposition~\ref{prop:key-1}, or employing \eqref{eq:ineq-grad}, we get
\begin{align*}\Exp \| P_t\nabla f(G)\|_2^2 \ls  \Exp
\| \nabla f(G) \|_2^2 R^{-\tanh t}
&\ls \Exp\| \nabla f(G) \|_2^2\, n^{-\delta\tanh t} \\
&\ls \Exp\| \nabla f(G) \|_2^2\, n^{-\delta t/e}.
\end{align*}
The result follows by a standard application of Markov's inequality.
\end{proof}

\begin{lemma}\label{l: aponalfafrha}
Let the norm $f(\cdot)$ satisfy $\delta\sqrt{\log n}\,\Lip(f)\ls \Exp f(G)$ and
$$\Prob\big\{|f(G)-\Exp f(G)|\gr s \Exp f(G)\big\}\ls 2\exp(-\delta s\log n),\quad s\gr 0,$$
for some parameter $\delta\in(0,1]$. Then for any $t\in[\log^{-1}n,1/2]$ we have
$$\Prob\big\{\langle G, P_t \nabla f(G)\rangle < (1-4t)\,\Exp f(G)\big\}\ls 2\exp(-c\delta\sqrt{t\log n})$$
for a universal constant $c>0$.
\end{lemma}
\begin{proof}
We have
\begin{align*}
\Prob&\big\{\langle G, P_t \nabla f(G)\rangle < (1-4t)\,\Exp f(G)\big\}\\
&\ls \Prob\big\{\langle G, P_t \nabla f(G)\rangle < (1-t)f(G)-2t\,\Exp f(G)\big\}\\
&\hspace{1cm}+\Prob\big\{(1-t)f(G)-2t\,\Exp f(G)<(1-4t)\,\Exp f(G)\big\}.
\end{align*}
Applying Lemma~\ref{l: apon,mfpauh}, we get
$$\Prob\big\{\langle G, P_t \nabla f(G)\rangle < (1-t)f(G)-2t\,\Exp f(G)\big\}\ls 2\exp\big(-c \delta \sqrt{t\log n}\big)$$
for some universal constant $c>0$.
On the other hand, the assumptions of the lemma imply
$$\Prob\big\{(1-t)f(G)-2t\,\Exp f(G)<(1-4t)\,\Exp f(G)\big\}\ls 2\exp(-c'\delta t\log n)$$
for a universal constant $c'>0$.
The result follows.
\end{proof}

\begin{proof}[Proof of Proposition~\ref{p: piul,ams}]
Without loss of generality, $\tau\ls 1/4$.
We have
\begin{align*}
\Exp \semi(G)&=\int\limits_0^\infty \Prob\big\{\semi(G)\gr s\big\}\,ds\\
&\gr 4\int\limits_{\tau}^{1/4} \Prob\big\{\semi(G)\gr (1-4s)\Exp f(G)\big\}\,\Exp f(G)\,ds\\
&=(1-4\tau)\Exp f(G)-4\int\limits_{\tau}^{1/4} \Prob\big\{\semi(G)< (1-4s)\Exp f(G)\big\}\,\Exp f(G)\,ds.
\end{align*}
Take any $s\in [\tau,1/4]$. Combining Lemmas~\ref{l: ooyeksn} and~\ref{l: aponalfafrha}
and the conditions on events $\Event_s$, we get
that the event $\{G\in F_s\setminus \Event_s\}$ (where the sets $F_\cdot$ are defined in \eqref{eq: F definition})
has probability at least
$1-2\exp(-c'\delta\sqrt{s\log n})$, for some constant $c'>0$.
By the definition of the seminorm $\semi(\cdot)$, we have
$\semi(G)\gr (1-4s)\Exp f(G)$ whenever $G\in F_s\setminus \Event_s$.
Hence,
$$\int\limits_{\tau}^{1/4} \Prob\big\{\semi(G)< (1-4s)\Exp f(G)\big\}\,ds
\ls 2\int\limits_{\tau}^{1/4}\exp(-c'\delta\sqrt{s\log n})\,ds\ls C\delta^{-2}\tau,$$
for some constant $C>0$. The result follows.
\end{proof}

\subsection{Seminorm deformation}
Let $T$ be non-empty closed bounded origin-symmetric subset of $\R^n\setminus\{0\}$,
and let $\textfrak{w}(\cdot)$ be a seminorm in $\R^n$ given by
$$\textfrak{w}(x):=\max\limits_{y\in T}\langle x,y\rangle,\quad x\in\R^n.$$
Further, assume that for some $i\ls n$ and $\alpha\in (0,1)$ we have
$$\Exp|\partial_i \textfrak{w}(G)|\gr \alpha\sqrt{\Exp\|\nabla \textfrak{w}(G)\|_2^2}.$$
We are interested in properties of a new seminorm $\widetilde{\textfrak{w}}(\cdot)$ obtained from
$\textfrak{w}(\cdot)$ by ``removing'' functionals with large scalar products with the $i$-th coordinate vector:
$$\widetilde{\textfrak{w}}(x):=\max\limits_{y\in T_i}\langle x,y\rangle,\quad x\in\R^n,$$
where
$$T_i:=\big\{y\in T:\;|y_i|\ls \mbox{$\frac{\alpha}{4}$}\sqrt{\Exp\|\nabla \textfrak{w}(G)\|_2^2}\big\}.$$
The main statement of the section is the following:
\begin{proposition}\label{p: piuh;alma}
Let $\textfrak{w}(\cdot)$, $\widetilde{\textfrak{w}}(\cdot)$, $\alpha\in (0,1)$ and $i\ls n$ be as above. Then, we have
$$\Exp \widetilde{\textfrak{w}}(G)\ls \Exp \textfrak{w}(G)-c\alpha^4\sqrt{\Exp\|\nabla \textfrak{w}(G)\|_2^2},$$
where $c>0$ is a universal constant.
\end{proposition}
\begin{remark}
The proposition can be interpreted as follows: if the partial derivative of the $i$--th coordinate of a seminorm
is large then removing the ``spikes'' in the definition of the set of norming functionals produces a seminorm
with expectation significantly less than the expectation of the original seminorm. In a sense, this implies that
the total number of such coordinates cannot be large. This will be very important in the next subsection
when we choose the events $(\Event_t)$ for the parametric family of seminorms $\semi(\cdot)$
and prove the main technical statement of the section.
\end{remark}

\begin{lemma}\label{l: suyrosamd;dsf}
Let $x\in\R^n$ be a vector such that the gradient $\partial_i \textfrak{w}(x)$ is well defined and let us assume that
$\partial_i \textfrak{w}(x)\gr \frac{\alpha}{2} \sqrt{\Exp\|\nabla \textfrak{w}(G)\|_2^2}$. Then, for any
$r>0$ one has
\[
\widetilde{\textfrak{w}}(x+re_i) \ls \textfrak{w}(x+re_i) - \frac{r\alpha}{4} \sqrt{ \mathbb E \|\nabla \textfrak{w}(G)\|_2^2}.
\]
\end{lemma}

\noindent {\it Proof.} Let $r>0$ and let $y\in T_i$ such that $\widetilde{\textfrak{w}}(x+re_i) = \langle x+re_i ,y\rangle$. It follows that
\begin{align} \label{eq:com-norm}
\widetilde{\textfrak{w}}(x+re_i) = \langle x,y\rangle + ry_i\ls \textfrak{w}(x) + \frac{r\alpha}{4} \sqrt{ \mathbb E\|\nabla \textfrak{w}(G)\|_2^2 },
\end{align} where we have used the definition of $T_i$ and the fact that $T_i \subseteq T$. 
On the other hand, the mapping $x_i \mapsto \textfrak{w}(x_1, \ldots, x_i , \ldots x_n)$ is convex
(when the other coordinates remain fixed), thus
\[
\textfrak{w}(x) \ls  \textfrak{w}(x+re_i) - r\partial_i \textfrak{w}(x) \ls \textfrak{w}(x+re_i) -\frac{r\alpha}{2} \sqrt{ \mathbb E\|\nabla \textfrak{w}(G)\|_2^2 }.
\] Inserting the latter into \eqref{eq:com-norm} we get the result. \prend

\begin{proof}[Proof of Proposition~\ref{p: piuh;alma}] The trivial bound 
$\mathbb E|\partial_i \textfrak{w}(G)|^2 \ls \mathbb E\|\nabla \textfrak{w}(G)\|_2^2$ and the Paley-Zygmund inequality \cite[p.47]{BLM} imply
\[
\mathbb P \left\{ |\partial_i \textfrak{w}(G)| \gr \frac{\alpha}{2} \sqrt{ \mathbb E\|\nabla \textfrak{w}(G)\|_2^2} \right\} \gr
\mathbb P\left\{ |\partial_i \textfrak{w}(G)| \gr \frac{1}{2} \mathbb E |\partial_i \textfrak{w}(G) | \right\} \gr 
\frac{1}{4} \frac{(\mathbb E|\partial_i \textfrak{w}(G)|)^2}{\mathbb E |\partial_i \textfrak{w}(G)|^2} \gr \frac{\alpha^2}{4}.
\] 
It follows that\footnote{Note that if 
$Q= \left \{ \partial_i\textfrak{w} \gr \tfrac{\alpha}{2}\sqrt{\mathbb E \|\nabla \textfrak{w}(G)\|_2^2 } \right\}$, 
then the evenness of $\textfrak{w}$ implies $-Q = \left \{- \partial_i\textfrak{w} \gr \tfrac{\alpha}{2}\sqrt{\mathbb E \|\nabla \textfrak{w}(G)\|_2^2 } \right\}$, and hence
$ \left \{|\partial_i\textfrak{w}| \gr \tfrac{\alpha}{2}\sqrt{\mathbb E \|\nabla \textfrak{w}(G)\|_2^2 } \right\} \subseteq Q\cup (-Q)$.}  the set
\[
Q:=\Big\{x: \partial_i \textfrak{w}(x)\gr \frac{\alpha}{2}\sqrt{\Exp\|\nabla \textfrak{w}(G)\|_2^2}\Big\},
\] satisfies $\gamma_n(Q)\gr \alpha^2/8$. We will need the following:

\smallskip

\noindent {\it Fact.} Let $Q \subset \mathbb R^n$. For any $0<\lambda <1$ and any $z\in \mathbb R^n$ one has
\[
\gamma_n(Q+z) \gr (\gamma_n(Q))^{1/\lambda} \exp\left( -\frac{\|z\|_2^2}{2(1-\lambda)} \right) .
\]

\smallskip

\noindent {\it Proof of Fact.} Let $q=\gamma_n(Q)$. First note that 
\[
\gamma_n(Q+z) = \int_{\mathbb R^n} {\bf 1}_Q e^{-\|z\|_2^2/2 - \langle x, z\rangle} \, d\gamma_n(x).
\] Hence, for any $0<\lambda<1$ we may apply H\"older's inequality to write
\begin{align*}
q = \int_{\mathbb R^n} \left( {\bf 1}_Q e^{- \frac{\|z\|_2^2}{2} -\langle x,z\rangle }\right)^\lambda 
\cdot e^{\lambda (\frac{\|x\|_2^2}{2} + \langle x, z\rangle )} \, d\gamma_n(x) 
& \ls (\gamma_n(Q+z))^\lambda \left(  \int_{\mathbb R^n}  e^{ \frac{\lambda}{1-\lambda} (\frac{\|x\|_2^2}{2} +\langle x, z\rangle) } \, d\gamma_n(x) \right)^{1-\lambda} \\
&= (\gamma_n(Q+z))^\lambda e^{\frac{\lambda}{2(1-\lambda)} \|z\|_2^2},
\end{align*} where we have used the rotation invariance of the Gaussian measure and the moment generating function of the standard normal. Reorganizing the above
inequality we arrive at the desired result. \prend


By Lemma~\ref{l: suyrosamd;dsf}, for any $y\in Q+ e_i$
we have
$$\widetilde{\textfrak{w}}(y)\ls \textfrak{w}(y)- \frac{\alpha}{4} \sqrt{\Exp\|\nabla \textfrak{w}(G)\|_2^2}.$$
This, together with the obvious relation $\widetilde{\textfrak{w}}(y)\ls \textfrak{w}(y)$, $y\in\R^n$ and the previous fact (for $\lambda=2/3$), gives
\begin{align*} 
\Exp \textfrak{w}(G)-\Exp\widetilde{\textfrak{w}}(G)
\gr \mathbb E[(\textfrak{w}(G)-\widetilde{\textfrak{w}}(G)) \mathbf 1_{Q+e_i}] 
&\gr \frac{ \alpha \gamma_n(Q)^{3/2} }{4e^{3/2}  }  \sqrt{\Exp\|\nabla \textfrak{w}(G)\|_2^2}.
\end{align*}
The bound $\gamma_n(Q)\gr \alpha^2/8$ yields the estimate
\[
\mathbb E\textfrak{w}(G) - \mathbb E\widetilde{\textfrak{w}}(G)
\gr c\alpha^4 \sqrt{\mathbb E\|\nabla \textfrak{w}(G)\|_2^2},
\] as required.
\end{proof}

\subsection{Balancing the seminorm: improving bounds on the variance}
Our bound on the Lipschitz constant of a seminorm $\semi(\cdot)$ from Proposition~\ref{p: piul,ams}
holds under very general (or no) assumptions on parameters $\delta,H$, the norm $f(\cdot)$ and events $(\Event_t)_{t\gr \tau}$,
but is not sufficiently strong to imply the main results of the section.
However, the bound on the variance of the seminorm can be improved provided that $H$ is ``sufficiently small''
and that the events $(\Event_t)_{t\gr \tau}$ are carefully chosen to guarantee that
the seminorm is ``well balanced''. It will be convenient to revise our notation a little.
Assume parameters $\tau,\delta,H$, and the norm $f$ are fixed.
Let $\mathcal F$ be the collection of all seminorms in $\R^n$ defined as in the Subsection~\ref{subs: seminorms}, where we allow all possible choices for the events $(\Event_t)_{t\gr \tau}$.
We will inductively construct a finite sequence of seminorms $(\semi_k)_{k=0}^m\subset\mathcal F$, where each seminorm corresponds to specially chosen events $(\Event_t^k)_{t\gr \tau}$,
starting with $\semi_0(\cdot)$ constructed for empty events $\Event_t^0:=\emptyset$, $t\gr \tau$.

Let $\theta>0$ be a parameter (later, we will connect it with paramters $H$ and $\delta$),
and assume that for some $k\gr 1$, the seminorm $\semi_{k-1}(\cdot)$ has been constructed.
If for every $i\ls n$ we have $\Exp|\partial_i \semi_{k-1}(G)|< \theta\Exp\semi_{0}(G)$,
then we set $m:=k-1$ and stop the construction. Otherwise, choose an index $i\ls n$
with $\Exp|\partial_i \semi_{k-1}(G)|\gr \theta\Exp\semi_{0}(G)$,
and for any $t\gr \tau$ let
$$\Event_t^k:=\Event_t^{k-1}\cup\big\{y\in\R^n:\;
|\langle e_i,P_t \nabla f(y)\rangle|\gr \mbox{$\frac{\theta}{4}$}\Exp\semi_{0}(G)\big\}.$$
Then we let $\semi_k(\cdot)$ to be the seminorm from $\mathcal F$ defined with parameters
$\tau,\delta,H$, $f$ and $(\Event_t^k)_{t\gr \tau}$, and proceed to the next step.

The following is a consequence of results of the previous subsection:
\begin{lemma}\label{l: oiugeojhalkf}
The sequence of seminorms constructed above is finite, with $m\ls C\theta^{-4}$,
and for any $k\ls m$
we have
$$\Exp\semi_{k}(G)\ls \Exp\semi_{k-1}(G)-c\theta^4\,\Exp \semi_{0}(G),$$
for universal constants $c,C>0$.
\end{lemma}
\begin{proof}
Take any two adjacent seminorms $\semi_{k-1}(\cdot)$ and $\semi_{k}(\cdot)$
from the sequence. According to the construction, there is an index $i\ls n$ such that
$\sqrt{\Exp\|\nabla \semi_{k-1}(G)\|_2^2}\gr\Exp|\partial_i \semi_{k-1}(G)|\gr \theta\Exp\semi_{0}(G)$.
Define $\alpha>0$ by $\alpha:=\theta\Exp\semi_{0}(G)/\sqrt{\Exp\|\nabla \semi_{k-1}(G)\|_2^2}$.
Note that the definition of the seminorm $\semi_{k}(\cdot)$
implies that every norming functional of $\semi_{k}(\cdot)$ is a limit of vectors
$P_t \nabla f(y)$ for some $y\in F_t\setminus \Event_t^k$
or $\frac{1}{H}(\epsilon_j P_t \partial_j f(y))_{j=1}^n$ for some $y\in F_t\setminus \Event_t^k$
and $\epsilon_j\in\{-1,1\}$, $j\ls n$.
Together with the definition of the events $\Event_t^k$, $t\gr \tau$, this implies that every norming functional $v=(v_1,\dots,v_n)$
of $\semi_{k}(\cdot)$ satisfies $|v_i|\ls \mbox{$\frac{\theta}{4}$}\Exp\semi_{0}(G)=\frac{\alpha}{4}
\sqrt{\Exp\|\nabla \semi_{k-1}(G)\|_2^2}$.

Then, by the definition of $\semi_{k}(\cdot)$ and by Proposition~\ref{p: piuh;alma},
we get
\begin{align*}
\Exp \semi_{k}(G)\ls \Exp \semi_{k-1}(G)-c'\alpha^4\sqrt{\Exp\|\nabla \semi_{k-1}(G)\|_2^2}
&=\Exp \semi_{k-1}(G)-c'\theta^4\,\Exp\semi_{0}(G)\,\bigg(\frac{\Exp\semi_{0}(G)}{\sqrt{\Exp\|\nabla \semi_{k-1}(G)\|_2^2}}\bigg)^3\\
&\ls \Exp \semi_{k-1}(G)-\widetilde c\theta^4\,\Exp \semi_{0}(G)
\end{align*}
for a universal constant $\widetilde c>0$, where in the last inequality we used that, for the standard Gaussian vector $Z$ independent from $G$ and some universal constant $\hat C>0$,
$$\hat C\,\Exp\semi_{0}(G)\gr \sqrt{\Exp\semi_{0}^2(G)}\gr \sqrt{\Exp_G\Exp_Z|\langle Z,\nabla \semi_{k-1}(G)\rangle|^2}
=\sqrt{\Exp_G\|\nabla \semi_{k-1}(G)\|_2^2}.$$

A recursive application of the bound gives
$$\Exp \semi_{k}(G)\ls (1-\widetilde c\theta^4 k)\Exp \semi_{0}(G),\quad k\gr 0,$$
whence
$m\ls \frac{1}{\widetilde c\theta^4}.$
\end{proof}
Next, we will estimate the Gaussian measure of events $\Event^k_t$ constructed above.
\begin{lemma}\label{l: ouakjnlfkgfoakf}
Let the norm $f(\cdot)$ satisfy condition \eqref{eq: ndelta cond}, and assume that $\delta\sqrt{\log n}\,\Lip(f)\ls \Exp f(G)$ and
$$\Prob\big\{|f(G)-\Exp f(G)|\gr s\, \Exp f(G)\big\}\ls 2\exp(-\delta s\log n),\quad s\gr 0.$$
Further, assume that $\log^{-1} n\ls \tau\ls \delta^2/(2C_{\ref{p: piul,ams}})$,
where the constant $C_{\ref{p: piul,ams}}$ is taken from Proposition~\ref{p: piul,ams}.
Then
for any $k\ls m$ and $t\gr \tau$ we have
$$\gamma_n(\Event^k_t)\ls \frac{Ck}{\theta}n^{-\delta/2},$$
where $C>0$ is a universal constant.
\end{lemma}
\begin{proof}
It is sufficient to show that for any given $i\ls n$, we have
$$q:=\gamma_n\Big\{y\in\R^n:\;|P_t\partial_i f(y)|\gr\frac{\theta}{4}\Exp\semi_0(G)\Big\}\ls \frac{C}{\theta}n^{-\delta/2}.$$
By our conditions on parameters and in view of Proposition~\ref{p: piul,ams},
$$\Exp\semi_0(G)\gr \frac{1}{2}\Exp f(G).$$
Hence, 
$$\Exp|\partial_i f(G)|\gr \Exp|P_t\partial_i f(G)|\gr \frac{\theta q}{8}\Exp f(G)\gr c\theta q \sqrt{\Exp\|\nabla f(G)\|_2^2}.$$
This, together with condition \eqref{eq: ndelta cond}, implies the bound.
\end{proof}

As a combination of Lemma~\ref{l: apiurls;,m}, Proposition~\ref{p: piul,ams}
and the above Lemmas~\ref{l: oiugeojhalkf} and~\ref{l: ouakjnlfkgfoakf}, we obtain
\begin{proposition}\label{p: apoh;d,masf}
There is a universal constant $C_{\ref{p: apoh;d,masf}}>0$ with the following property.
Let $n\gr C_{\ref{p: apoh;d,masf}}$; let the norm $f(\cdot)$ and parameter $\delta$ be as in Lemma~\ref{l: ouakjnlfkgfoakf}.
Assume additionally that $\log^{-1} n\ls \tau\ls \delta^2/(2C_{\ref{p: apoh;d,masf}})$ and that $\theta^5\gr n^{-\delta/4}$.
Then the seminorm $\semi_m(\cdot)$ constructed above satisfies
\begin{itemize}
\item $\semi_m(x)\ls f(x)$ for all $x\in\R^n$;
\item $\Exp\semi_m(G)\gr (1-C_{\ref{p: apoh;d,masf}}\delta^{-2}\tau)\Exp f(G)$;
\item $\Lip(\semi_m)\ls \Lip(f)\,n^{-\delta \tau/8}$;
\item $\Exp|\partial_i(\semi_m(G))|\ls 2\theta\Exp\semi_{m}(G)$ for all $i\ls n$.
\end{itemize}
\end{proposition}
\begin{proof}
The assertions $\semi_m(x)\ls f(x)$ and $\Lip(\semi_m)\ls \Lip(f)\,n^{-\delta \tau/8}$ are a direct application of Lemma~\ref{l: apiurls;,m}.
Further, the restriction on the parameter $\theta$ together with Lemmas~\ref{l: oiugeojhalkf} and~\ref{l: ouakjnlfkgfoakf}, imply that
$\gamma_n(\Event^m_t)\ls \frac{\hat C}{\theta^5}n^{-\delta/2}\ls \hat C n^{-\delta/4}$.
As long as $C_{\ref{p: apoh;d,masf}}>0$ is chosen sufficiently large, the assumption $\delta^2\gr (\log n)^{-1}$ and $n\gr C_{\ref{p: apoh;d,masf}}$ imply that
$\hat C n^{-\delta/4}\ls 2 n^{-\delta/8}\ls 2 n^{-\delta t/4}$.
Therefore, as long as $C_{\ref{p: apoh;d,masf}}\gr C_{\ref{p: piul,ams}}$, Proposition~\ref{p: piul,ams} can be applied to yield
$\Exp\semi_m(G)\gr (1-C_{\ref{p: apoh;d,masf}}\delta^{-2}\tau)\Exp\, f(G)\gr \frac{1}{2}\Exp\, f(G)$.
Finally, by the construction of the sequence $(\semi_k)_{k=0}^m$, we have
$\Exp|\partial_i \semi_{k-1}(G)|< \theta\,\Exp\semi_{0}(G)\ls \theta\, \Exp\, f(G)$, where the last quantity is in turn bounded above by $2\theta\,\Exp\semi_{m}(G)$. 
\end{proof}

Finally, we can prove the main technical result of the section. For convenience,
we will explicitly restate all conditions on the parameters involved.
\begin{theorem}\label{th: smalldev main}
There are universal constants $C_{\ref{th: smalldev main}},c_{\ref{th: smalldev main}}>0$ with the following property.
Let $\delta\in(0,1]$, let $n\gr C_{\ref{th: smalldev main}}$, and let $f(\cdot)$ be a norm in $\R^n$ such that $n^{\delta/64}\gr H:=\unc(f,\{e_i\}_{i=1}^n)$.
Further, take $\tau$ satisfying $\log^{-1} n\ls \tau\ls \delta^2/(2C_{\ref{th: smalldev main}})$.
Assume that \eqref{eq: ndelta cond} is satisfied, that $\delta\sqrt{\log n}\,\Lip(f)\ls \Exp f(G)$ and
$$\Prob\big\{|f(G)-\Exp f(G)|\gr s\, \Exp f(G)\big\}\ls 2\exp(-\delta s\log n),\quad s\gr 0.$$
Then 
$$\Prob\big\{f(G)\ls (1-C_{\ref{th: smalldev main}}\delta^{-2}\tau)\Exp f(G)\big\}
\ls 3\exp\big(-n^{c_{\ref{th: smalldev main}}\delta\tau}\big).$$
\end{theorem}
\begin{proof}
We will assume that $C_{\ref{th: smalldev main}}\gr 2C_{\ref{p: apoh;d,masf}}$.
Let $\theta:=n^{-\delta/32}$, and let $\semi_m(\cdot)$ be the seminorm constructed above.
We have $\semi_m(\cdot)\ls f(\cdot)$ and
$$\Exp \semi_m(G)\gr (1-C_{\ref{p: apoh;d,masf}}\delta^{-2}\tau)\Exp f(G),$$
by Proposition~\ref{p: apoh;d,masf}. This implies
$$\Prob\big\{f(G)\ls (1-C_{\ref{th: smalldev main}}\delta^{-2}\tau)\Exp f(G)\big\}
\ls \Prob\big\{\semi_m(G)\ls
(1-C_{\ref{p: apoh;d,masf}}\delta^{-2}\tau)\Exp\semi_m(G)\big\}.$$
Next, we will estimate the variance $\Var \semi_m(G)$.
As a corollary of Talagrand's $L^1-L^2$ bound (see Corollary~\ref{cor:bd-s-R}), we obtain
\begin{equation}\label{eq: 29857525093-0857-2}
\Var\semi_m(G)\ls \frac{\widetilde C\Exp \|\nabla \semi_m(G)\|_2^2}{1+
\log \frac{ \Exp\|\nabla \semi_m(G)\|_2^2}{ \sum_{i=1}^n \|\partial_i \semi_m\|_{L^1(\gamma_n)}^2 }},
\end{equation}
for a universal constant $\widetilde C>0$.
Clearly, $\Exp \|\nabla \semi_m(G)\|_2^2\ls \Lip(\semi_m)^2\ls n^{-\delta \tau/4}
\Lip(f)^2$, in view of Proposition~\ref{p: apoh;d,masf}.
Further, by the same proposition we have for any $i\ls n$:
$$\|\partial_i \semi_m\|_{L^1(\gamma_n)}
\ls 2\theta \Exp \semi_{m}(G).$$
Note that if $Z$ is the standard Gaussian vector independent from $G$ then
\begin{equation}\label{019874019874098}
\begin{split}
\sum\limits_{i=1}^n \|\partial_i \semi_m\|_{L^1(\gamma_n)}&=\sqrt{\pi/2}\,\Exp_G\Exp_Z\,\sum_{i=1}^n |G_i\cdot \partial_i\semi_m(Z)|\\
&=\sqrt{\pi/2}\,\Exp_G\Exp_Z\big\langle G,({\rm sign}(G_i)\,\partial_i\semi_m(Z))_{i=1}^n\big\rangle
\ls \sqrt{\pi/2}\,H\,\Exp \semi_{m}(G),
\end{split}
\end{equation}
where we have used that, by the $H$--unconditionality of the standard vector basis with respect to $\semi_m(\cdot)$,
the vector $({\rm sign}(G_i)\,\partial_i\semi_m(Z))_{i=1}^n$ belongs to $H$--enlargement of the unit dual ball for $\semi_m(\cdot)$.
This implies
$$\sum_{i=1}^n \|\partial_i \semi_m\|_{L^1(\gamma_n)}^2
\ls \max\limits_{j\ls n}\|\partial_j \semi_m\|_{L^1(\gamma_n)}\,
\Big(\sum\limits_{i=1}^n \|\partial_i \semi_m\|_{L^1(\gamma_n)}\Big)
\ls C\theta H(\Exp \semi_{m}(G))^2\ls Cn^{-\delta/64}(\Exp \semi_{m}(G))^2.$$
Now, if $(\Exp \semi_m(G))^2< n^{\delta/128}\Exp\|\nabla \semi_m(G)\|_2^2$
then the last estimate gives
$$\sum_{i=1}^n \|\partial_i \semi_m\|_{L^1(\gamma_n)}^2\ls Cn^{-\delta/128}\Exp\|\nabla \semi_m(G)\|_2^2,$$
so, \eqref{eq: 29857525093-0857-2} implies
$$\Var\semi_m(G)\ls \frac{C'n^{-\delta \tau/4}\Lip(f)^2}{1+\delta\log n}
\ls \frac{C'n^{-\delta \tau/4}\delta^{-2}(\Exp f(G))^2}{\log n+\delta\log^2 n}.$$
On the other hand, if $(\Exp \semi_m(G))^2\gr n^{\delta/128}\Exp\|\nabla \semi_m(G)\|_2^2$
then the Poincar\'e inequality immediately implies a bound
$$\Var\semi_m(G)\ls 2n^{-c\delta}(\Exp \semi_m(G))^2
\ls 2n^{-c\delta}(\Exp f(G))^2
.$$

As the final step of the proof, we observe that, according to Theorem~\ref{thm:sdi},
we have
$$\Prob\big\{\semi_m(G)\ls
(1-C_{\ref{p: apoh;d,masf}}\delta^{-2}\tau)\Exp\semi_m(G)\big\}
\ls 2\exp\big(-c\delta^{-4}\tau^2(\Exp\semi_m(G))^2/\Var\semi_m(G)\big).$$
The result follows from the upper bound on the variance obtained above.
\end{proof}

\subsection{Proof of Theorems~\ref{th: smdev gen} and~\ref{th: smdev unc}}

\begin{proof}[Proof of Theorem~\ref{th: smdev unc}]
Clearly, it is sufficient to show that any $1$--unconditional norm $f(\cdot)$ in the position of minimal $M$, 
or $\ell$--position, or $w^{1,1}$--position
satisfies the conditions of Theorem~\ref{th: smalldev main} with $\delta$ being a universal constant.
We consider the position of minimal $M$ first.
The conditions $\delta\sqrt{\log n}\,\Lip(f)\ls \Exp f(G)$ and
$\Prob\big\{|f(G)-\Exp f(G)|\gr s\, \Exp f(G)\big\}\ls 2\exp(-\delta s\log n),\quad s\gr 0$
(for constant $\delta$) are known; see, in particular, \cite[Proposition~2.5]{PV-dicho}.
If, additionally, \eqref{eq: ndelta cond} holds then a direct application of Theorem~\ref{th: smalldev main}
gives the result.
On the other hand, if \eqref{eq: ndelta cond} does not hold, say, for $\delta=1/10$, then,
using the balancing conditions $\Exp(G_i\partial_i f(G))=\Exp|G_i\partial_i f(G)|=\frac{1}{n}\Exp f(G)$, $i\ls n$
for the position of minimal $M$, and the Poincar\'e inequality, we obtain
$$
\Var f(G)\ls \Exp \|\nabla f(G)\|_2^2\ls n^{1/10}\sum_{i=1}^n (\Exp | \partial_i f(G) |)^2 \ls \frac{\pi}{2}\,n^{1/10}\sum_{i=1}^n (\Exp | G_i \partial_i f(G) |)^2 =\frac{\pi}{2}\,n^{-9/10}(\Exp f(G))^2,
$$
where we used that for any $1$--unconditional seminorm for every coordinate $i$ the quantities $|G_i|$ and $|\partial_i f(G)|$ are positively correlated (see, in particular, \cite[Section~2.3]{Tik-unc}
for a deterministic statement for $1$--uncondi\-tional norms which implies the positive correlation).
Then Theorem~\ref{thm:sdi} implies the required bound.

For the $\ell$--position, the conditions $\delta\sqrt{\log n}\,\Lip(f)\ls \Exp f(G)$ and
$\Prob\big\{|f(G)-\Exp f(G)|\gr s\, \Exp f(G)\big\}\ls 2\exp(-\delta s\log n),\quad s\gr 0$ are known as well (see \cite{Tik-unc}).
Again, if \eqref{eq: ndelta cond} holds for a constant positive $\delta$ then we just apply Theorem~\ref{th: smalldev main}.
Otherwise, using the balancing conditions $\Exp(|G_i|\,f(G)\,|\partial_i f(G)|)=\frac{1}{n}\Exp (f(G)^2)$, $i\ls n$, and the Poincar\'e inequality, we obtain $\Var f(G)\ls n^{-c}(\Exp f(G))^2$
(see, in particular, \cite{Tik-unc} or \cite[Note~2.6]{PV-dicho}), and invoke Theorem~\ref{thm:sdi}.

The argument for $w^{1,1}$--position almost literally repeats the above, up to using the appropriate balancing conditions.
\end{proof}

\begin{proof}[Proof of Theorem~\ref{th: smdev gen}]
Apart from applying Theorem~\ref{th: smalldev main}, the proof essentially
reproduces the ideas from Subsection~\ref{subs: gen case}, so we only sketch the argument.
The starting point is the dichotomy employed earlier in paper \cite{PV-dicho} and, before that, introduced in \cite{Sch06}:
given any norm $f(\cdot)$ in $\R^n$ with the unit ball in John's position, and any $c\in(0,1/4]$, either
the critical dimension $k(f)$ is greater than $n^{c}$, or, as a corollary of Alon--Milman's
theorem \cite{AM}, there is a subspace $X:=(E,f(\cdot))$
of $(\R^n,f(\cdot))$ of dimension $c'n^{1/2}$ which has a $Cn^{c/2}$--unconditional basis (see, for example, proof of Theorem~4.4 in \cite{PV-dicho}).
Here, $c',C>0$ are constants which may only depend on $c$.
We will choose $c>0$ later.

In the former case,
the required result immediately follows from standard concentration estimates for Lipschitz functions in the Gauss space.
In the latter case, applying an appropriate linear transformation to $X$ (whose existence is guaranteed by the Borsuk--Ulam
theorem, see \cite[Lemma 3.1]{PV-dicho}),
we obtain
a normed space $(\R^{c'n^{1/2}},\widetilde f(\cdot))$ such that the standard basis $\{e_i\}_{i=1}^{c'n^{1/2}}$
is $Cn^{c/2}$--unconditional, and, moreover,
$\Exp|\partial_i \widetilde f(G')|=\Exp|\partial_j \widetilde f(G')|$ for all $i\neq j$, where $G'$ is the standard 
Gaussian vector in $\R^{c'n^{1/2}}$.
The $Cn^{c/2}$--unconditionality of the standard basis then implies, similarly to the argument in \eqref{019874019874098},
\begin{align*}
\sum\limits_{i=1}^{c'n^{1/2}}\big(\Exp|\partial_i \widetilde f(G')|\big)^2
&\ls C'n^{c/2}\Exp \widetilde f(G') \Exp|\partial_1 \widetilde f(G')|\\
&=
C''n^{c/2-1/2}\Exp \widetilde f(G') \sum\limits_{i=1}^{c'n^{1/2}}\Exp|\partial_i \widetilde f(G')|\\
&\ls C''' n^{c-1/2}(\Exp \widetilde f(G'))^2.
\end{align*}
If $(\Exp \widetilde f(G'))^2\ls n^{c}\Exp\|\nabla \widetilde f(G')\|_2^2$ and
assuming that the constant $c>0$ is sufficiently small, the norm $\widetilde f(\cdot)$
satisfies assumptions of Theorem~\ref{th: smalldev main}
with $c'n^{1/2}$ in place of $n$, and with $\delta:=256c$,
implying the lower deviation estimates for $\widetilde f(\cdot)$
(the conditions $\delta\sqrt{\log n}\,\Lip(\widetilde f)\ls \Exp \widetilde f(G')$ and
$\Prob\big\{|\widetilde f(G')-\Exp \widetilde f(G')|\gr s\, \Exp \widetilde f(G')\big\}
\ls 2\exp(-\delta s\log n),\quad s\gr 0$, follow from the argument in the proof of Theorem~3.4 of \cite{PV-dicho}).
On the other hand, if $(\Exp \widetilde f(G'))^2\gr n^{c}\,\Exp\|\nabla \widetilde f(G')\|_2^2$
then the deviation estimates follow immediately from Theorem~\ref{thm:sdh}. 
Finally, we apply the argument from \cite{PV-dicho} (see the proof of Theorem 4.5 there)
to ``lift'' the lower deviation estimates for $\widetilde f(\cdot)$ back to an appropriate linear transformation of $(\R^n,f(\cdot))$.
\end{proof}

\bigskip
\bibliography{hs-sbe-ref}
\bibliographystyle{alpha}


\end{document}